\documentclass[10pt]{amsart}
\usepackage{amssymb}
\usepackage{bm}
\usepackage{graphicx}
\usepackage[centertags]{amsmath}
\usepackage{amsfonts}
\usepackage{amsthm}
\usepackage{graphicx}
\newtheorem{thm}{Theorem}
\newtheorem{cor}[thm]{Corollary}
\newtheorem{lem}[thm]{Lemma}

\newtheorem{claim}[thm]{Claim}
\newtheorem{fact}[thm]{Fact}

\newtheorem{defn}[thm]{Definition}
\theoremstyle{definition}

\newcommand{\nn}{\mathbb{N}}
\newcommand{\ee}{\varepsilon}

\newcommand{\suc}{\mathrm{Succ}}
\newcommand{\immsuc}{\mathrm{ImmSucc}}
\newcommand{\strong}{\mathrm{Str}}
\newcommand{\ci}{\mathrm{I}}
\newcommand{\mil}{\mathrm{Mil}}

\newcommand{\dhl}{\mathrm{DHL}}
\newcommand{\udhl}{\mathrm{UDHL}}
\newcommand{\dens}{\mathrm{dens}}
\newcommand{\ave}{\mathbb{E}}
\newcommand{\ls}{\mathrm{LS}}

\newcommand{\bfr}{\mathbf{r}}
\newcommand{\bfs}{\mathbf{s}}
\newcommand{\bft}{\mathbf{t}}
\newcommand{\bfu}{\mathbf{u}}
\newcommand{\bfv}{\mathbf{v}}

\newcommand{\bfz}{\mathbf{z}}

\newcommand{\bfcb}{\mathbf{B}}

\newcommand{\bfcf}{\mathbf{F}}
\newcommand{\bfcg}{\mathbf{G}}

\newcommand{\bfci}{\mathbf{I}}

\newcommand{\bfcq}{\mathbf{Q}}
\newcommand{\bfcr}{\mathbf{R}}
\newcommand{\bfcs}{\mathbf{S}}
\newcommand{\bfct}{\mathbf{T}}
\newcommand{\bfcu}{\mathbf{U}}
\newcommand{\bfcv}{\mathbf{V}}
\newcommand{\bfcz}{\mathbf{Z}}

\newcommand{\meg}{\geqslant}
\newcommand{\mik}{\leqslant}

\newcommand{\con}{\smallfrown}

\begin{document}

\title{Dense subsets of products of finite trees}

\author{Pandelis Dodos, Vassilis Kanellopoulos and Konstantinos Tyros}

\address{Department of Mathematics, University of Athens, Panepistimiopolis 157 84, Athens, Greece}
\email{pdodos@math.uoa.gr}

\address{National Technical University of Athens, Faculty of Applied Sciences,
Department of Mathematics, Zografou Campus, 157 80, Athens, Greece}
\email{bkanel@math.ntua.gr}

\address{Department of Mathematics, University of Toronto, Toronto, Canada M5S 2E4}
\email{k.tyros@utoronto.ca}

\thanks{The first named author was supported by NSF grant DMS-0903558.}

\maketitle


\begin{abstract}
We prove a ``uniform" version of the finite density Halpern--L\"{a}uchli Theorem. Specifically, we say that a tree $T$ is homogeneous
if it is uniquely rooted and there is an integer $b\meg 2$, called the branching number of $T$, such that every $t\in T$ has
exactly $b$ immediate successors. We show the following.
\medskip

\noindent \textit{For every integer $d\meg 1$, every $b_1,...,b_d\in\nn$ with $b_i\meg 2$ for all $i\in\{1,...,d\}$, every integer $k\meg 1$
and every real $0<\ee\mik 1$ there exists an integer $N$ with the following property. If $(T_1,...,T_d)$ are homogeneous trees such that the
branching number of $T_i$ is $b_i$ for all $i\in\{1,...,d\}$, $L$ is a finite subset of $\nn$ of cardinality at least $N$ and $D$ is a subset
of the level product of $(T_1,...,T_d)$ satisfying
\[|D\cap \big(T_1(n)\times ...\times T_d(n)\big)| \meg \ee |T_1(n)\times ...\times T_d(n)|\]
for every $n\in L$, then there exist strong subtrees $(S_1,...,S_d)$ of $(T_1,...,T_d)$ of height $k$ and with common level set such that
the level product of $(S_1,...,S_d)$ is contained in $D$. The least integer $N$ with this property will be denoted by
$\mathrm{UDHL}(b_1,...,b_d|k,\ee)$.}
\medskip

The main point is that the result is independent of the position of the finite set $L$. The proof is based on a density increment strategy and
gives explicit upper bounds for the numbers $\mathrm{UDHL}(b_1,...,b_d|k,\ee)$.
\end{abstract}


\section{Introduction}

\subsection{Statement of the problem and the main result}

It is well-known that several results in Ramsey Theory have an infinite and a finite version. While the proofs of the infinite versions are
usually conceptually ``cleaner" and yield formally stronger results (see, e.g., \cite{PH}), an analysis of the corresponding finite versions
gives explicit and non-trivial estimates for certain numerical invariants commonly known as \textit{Ramsey numbers}. These invariants
are of fundamental importance and are central for the development of Ramsey Theory (see \cite{GRS}).

The main goal of the present paper is to give an effective proof of a ``uniform" version of the finite density Halpern--L\"{a}uchli Theorem
and to obtain quantitative information on the corresponding ``density Halpern--L\"{a}uchli" numbers. To proceed with our discussion it is useful
at this point to recall the Halpern--L\"{a}uchli Theorem \cite{HL}. It is a rather deep pigeonhole principle for trees. It has several equivalent
forms which are discussed in great detail in \cite[\S 3.1]{To2}. We will state the ``strong subtree version" which is the most important one from
a combinatorial perspective.
\begin{thm} \label{t11}
For every integer $d\meg 1$, every tuple $(T_1,...,T_d)$ of uniquely rooted and finitely branching trees without maximal nodes and every finite
coloring of the level product
\[  \bigcup_{n\in\nn} T_1(n)\times ...\times T_d(n) \]
of $(T_1,...,T_d)$ there exist strong subtrees $(S_1, ..., S_d)$ of $(T_1,...,T_d)$ of infinite height and with common level set such that the
level product of $(S_1,...,S_d)$ is monochromatic.
\end{thm}
We recall that a subtree $S$ of a tree $(T,<)$ is said to be \textit{strong} if: (a) $S$ is uniquely rooted and balanced (that is, all maximal
chains of $S$ have the same cardinality), (b) every level of $S$ is a subset of some level of $T$, and (c) for every non-maximal node $s\in S$
and every immediate successor $t$ of $s$ in $T$ there exists a unique immediate successor $s'$ of $s$ in $S$ with $t\mik s'$. The \textit{level
set} of a strong subtree $S$ of a tree $T$ is the set of levels of $T$ containing a node of $S$.

Although the notion of a strong subtree was isolated in the 1960s, it was highlighted with the work of K. Milliken in \cite{Mi1,Mi2} who
used Theorem \ref{t11} to show that the family of strong subtrees of a uniquely rooted and finitely branching tree is partition regular.
The Halpern--L\"{a}uchli Theorem and Milliken's Theorem can be considered as the starting point of Ramsey Theory for trees, a rich area of
Combinatorics with significant applications, most notably in the Geometry of Banach spaces (see, for instance, \cite{Bl,C,CS,G,Ka,La,LSV,To2}
and \cite{ADK,D,Stern,To1} for applications).

Theorem \ref{t11} has a density version that was conjectured by R. Laver in the late 1960s and obtained, recently, in \cite{DKK}. To state it
let us recall that a tree $T$ is said to be \textit{homogeneous} if it is uniquely rooted and there exists an integer $b\meg 2$, called the
\textit{branching number} of $T$, such that every $t\in T$ has exactly $b$ immediate successors.
\begin{thm} \label{t12}
For every integer $d\meg 1$, every tuple $(T_1,...,T_d)$ of homogeneous trees and every subset $D$ of the level product of $(T_1,...,T_d)$
satisfying
\[ \limsup_{n\to\infty} \frac{|D\cap \big( T_1(n)\times ... \times T_d(n)\big)|}{|T_1(n)\times ... \times T_d(n)|}>0 \]
there exist strong subtrees $(S_1, ..., S_d)$ of $(T_1,...,T_d)$ of infinite height and with common level set such that
the level product of $(S_1,...,S_d)$ is contained in $D$.
\end{thm}
We should point out that the assumption in Theorem \ref{t12} that the trees $(T_1,...,T_d)$ are homogeneous is not redundant. On the contrary,
various examples given in \cite{BV} show that it is essentially optimal.

While Theorem \ref{t12} is infinite-dimensional, it has a finite counterpart which is obtained via a standard compactness argument. Precisely,
it follows by Theorem \ref{t12} that for every integer $d\meg 1$, every $b_1,...,b_d\in\nn$ with $b_i\meg 2$ for all $i\in\{1,...,d\}$, every
integer $k\meg 1$, every real $0<\ee\mik 1$ and every $M=\{m_0<m_1<...\}$ infinite subset of $\nn$, there exists an integer $N$ with the following
property. If $(T_1,...,T_d)$ is a tuple of homogeneous trees such that the branching number of $T_i$ is $b_i$ for all $i\in\{1,...,d\}$ and $D$
is a subset of the level product of $(T_1,...,T_d)$ satisfying
\[ |D\cap \big(T_1(m_n)\times ...\times T_d(m_n)\big)| \meg \ee |T_1(m_n)\times ...\times T_d(m_n)| \]
for every $n\mik N$, then there exist strong subtrees $(S_1, ..., S_d)$ of $(T_1,...,T_d)$ of height $k$ and with common level set such that
the level product of $(S_1,...,S_d)$ is contained in $D$. We shall denote the least integer $N$ with this property by $\dhl(b_1,...,b_d|k,\ee,M)$.
We emphasize that the compactness argument yields that the integer $\dhl(b_1,...,b_d|k,\ee,M)$ depends on the choice of the infinite set $M$.

There are two basic problems left open by the previous approach. The first one is to provide explicit upper bounds for the numbers $\dhl(b_1,...,b_d|k,\ee,M)$. This is, of course, related to the ineffectiveness of the compactness method. The second problem lies deeper and
concerns the uniform boundedness of the numbers $\dhl(b_1,...,b_d|k,\ee,M)$ with respect to the last parameter. Precisely, if the parameters
$b_1,...,b_d,k$ and $\ee$ are fixed, then does there exist an integer $j$ such that $\dhl(b_1,...,b_d|k,\ee,M)\mik j$ for every infinite subset
$M$ of $\nn$? In other words, is it true that a finite subset of the level product of a fixed tuple of homogeneous trees will necessarily contain
a ``substructure" as long as it is dense in sufficiently many levels? This information is needed in various applications
(\cite{DKT}) and constitutes the proper finite analog of Theorem \ref{t12}.

Our main result answers the above questions. Precisely, we show the following.
\begin{thm} \label{t13}
For every integer $d\meg 1$, every $b_1,...,b_d\in\nn$ with $b_i\meg 2$ for all $i\in\{1,...,d\}$, every integer $k\meg 1$ and every real
$0<\ee\mik 1$ there exists an integer $N$ with the following property. If $(T_1,...,T_d)$ are homogeneous trees such that the branching number
of $T_i$ is $b_i$ for all $i\in\{1,...,d\}$, $L$ is a finite subset of $\nn$ of cardinality at least $N$ and $D$ is a subset of the level product
of $(T_1,...,T_d)$ satisfying
\[ |D\cap \big(T_1(n)\times ...\times T_d(n)\big)| \meg \ee |T_1(n)\times ...\times T_d(n)| \]
for every $n\in L$, then there exist strong subtrees $(S_1,...,S_d)$ of $(T_1,...,T_d)$ of height $k$ and with common level set such
that the level product of $(S_1,...,S_d)$ is contained in $D$. The least integer $N$ with this property will be denoted by $\udhl(b_1,...,b_d|k,\ee)$.
\end{thm}
As we have already mentioned, the proof of Theorem \ref{t13} is effective and gives explicit upper bounds for the numbers
$\udhl(b_1,...,b_d|k,\ee)$. These upper bounds are admittedly rather weak; they have an Ackermann-type dependence with respect
to the ``dimension" $d$. However, they are in line with several other bounds obtained recently in the area; see \cite{Gowers,Pol,RNSSK}.

\subsection{Related work}

There are several results in the literature closely related to the one-dimensional case of Theorem \ref{t13}, namely when we deal with
a single homogeneous tree. The earliest reference we are aware of is the paper \cite{BV} by R. Bicker and B. Voigt, though related problems
have been circulated among experts much earlier. The first significant progress, however, was made by H. Furstenberg and B. Weiss in \cite{FW} who
obtained a ``parameterized" version of Sz\'{e}merdi's Theorem on arithmetic progressions \cite{Sz}. Specifically, it was shown in \cite{FW}
that for every integer $b\meg2$, every integer $k\meg 1$ and every real $0<\ee\mik 1$ there exists an integer $N$ with the following property.
If $T$ is a finite homogeneous tree with branching number $b$ and of height at least $N$, $L$ is a subset of $\{0,...,h(T)-1\}$ of cardinality
at least $\ee h(T)$ and $D$ is a subset of $T$ satisfying
\[ |D\cap T(n)| \meg \ee |T(n)| \]
for every $n\in L$, then $D$ contains a strong subtree of $T$ of height $k$ whose level set is an arithmetic progression. We shall denote
by $\mathrm{FW}(b|k,\ee)$ the least integer $N$ with this property. The method in \cite{FW} was qualitative in nature, and as such, could
not provide explicit estimates for the numbers $\mathrm{FW}(b|k,\ee)$. The work of H. Furstenberg and B. Weiss was revisited, recently,
in \cite{PST} by J. Pach, J. Solymosi and G. Tardos who effectively reduced the aforementioned result to Sz\'{e}meredi's Theorem with an
elegant combinatorial argument. It follows, in particular, from the analysis in \cite{PST} that $\udhl(b|k,\ee)=O_{b,\ee}(k)$, an upper bound
which is essentially optimal.

The proof, however, of the higher-dimensional case of Theorem \ref{t13} follows quite different arguments and is closer in spirit to the
``polymath" proof \cite{Pol} of the density Hales--Jewett Theorem \cite{FK}. It proceeds by induction on the ``dimension" $d$ and is based
on a density increment strategy, a powerful and fruitful method pioneered by K. F. Roth \cite{Roth}.

\subsection{Organization of the paper}

The paper is organized as follows. In \S 2 we gather some background material. In \S 3 we introduce the concept of a \textit{level selection}.
It is a natural notion permitting us to reduce the proof of our main result to the study of certain dense subsets with special properties.
In \S 4 we give a detailed outline of the proof of Theorem \ref{t13} emphasizing, in particular, its main features. In the next three sections,
\S 5-\S 7, we prove several preparatory results. We notice that these sections are largely independent of each other and can be read separately.
This material is used in \S 8 which contains the last step of the argument. Finally, the proof of Theorem \ref{t13} is given in \S 9.

To facilitate the interested reader we have also included, in an appendix, a sketch of the proof of the multidimensional version of Milliken's
Theorem \cite{Mi2}. This result and the corresponding bounds are needed for the proof of Theorem \ref{t13}. The arguments are essentially borrowed
from \cite{So} and are included for completeness.


\section{Background material}

By $\nn=\{0,1,2,...\}$ we shall denote the natural numbers. The cardinality of a set $X$ will be denoted by $|X|$ while its powerset will
be denoted by $\mathcal{P}(X)$. If $X$ is a nonempty finite set, then by $\ave_{x\in X}$ we shall denote the average $\frac{1}{|X|} \sum_{x\in X}$.
If it is clear from the context to which set $X$ we are referring to, then this average will be denoted simply by $\ave_x$. For every function
$f:\nn\to\nn$ and every $k\in\nn$ by $f^{(k)}:\nn\to\nn$ we shall denote the $k$-th iteration of $f$ defined recursively by the rule
$f^{(0)}(n)=n$ and $f^{(k+1)}(n)=f\big(f^{(k)}(n)\big)$ for every $n\in\nn$.

\subsection{Trees}

By the term \textit{tree} we mean a nonempty partially ordered set $(T,<)$ such that the set $\{s\in T: s<t\}$ is finite and linearly
ordered under $<$ for every $t\in T$. The cardinality of this set is defined to be the \textit{length} of $t$ in $T$ and will be denoted by
$\ell_T(t)$. For every $n\in\nn$ the \textit{$n$-level} of $T$, denoted by $T(n)$, is defined to be the set $\{t\in T: \ell_T(t)=n\}$.
The \textit{height} of $T$, denoted by $h(T)$, is defined as follows. If there exists $k\in\nn$ with $T(k)=\varnothing$, then we set
$h(T)=\max\{n\in\nn: T(n)\neq\varnothing\}+1$; otherwise, we set $h(T)=\infty$.

For every node $t$ of a tree $T$ the set of \textit{successors} of $t$ in $T$ is defined by
\begin{equation} \label{e21}
\suc_T(t)=\{s\in T: t\mik s\}.
\end{equation}
The set of \textit{immediate successors} of $t$ in $T$ is the subset of $\suc_T(t)$ defined by
$\immsuc_T(t)=\{s\in T: t\mik s \text{ and } \ell_T(s)=\ell_T(t)+1\}$. A node $t\in T$ is said to be \textit{maximal}
if the set $\immsuc_T(t)$ is empty.

Let $n\in\nn$ with $n< h(T)$ and $F\subseteq T(n)$. The \textit{density} of $F$ is defined by
\begin{equation} \label{e22}
\dens(F)=\frac{|F|}{|T(n)|}.
\end{equation}
More generally, for every $m\in\nn$ with $m\mik n$ and every $t\in T(m)$ the \textit{density of $F$ relative to the node $t$}
is defined by
\begin{equation} \label{e23}
\dens(F \ | \ t)=\frac{|F\cap \suc_T(t)|}{|T(n)\cap \suc_T(t)|}.
\end{equation}

A \textit{subtree} $S$ of a tree $(T,<)$ is a subset of $T$ viewed as a tree equipped with the induced partial
ordering. For every $n\in\nn$ with $n<h(T)$ we set
\begin{equation} \label{e24}
T\upharpoonright n= T(0)\cup ... \cup T(n).
\end{equation}
Notice that $h(T\upharpoonright n)=n+1$. An \textit{initial subtree} of $T$ is a subtree of $T$ of the form
$T\upharpoonright n$ for some $n\in\nn$.

A tree $T$ is said to be \textit{balanced} if all maximal chains of $T$ have the same cardinality. It is said to be \textit{uniquely rooted}
if $|T(0)|=1$; the \textit{root} of a uniquely rooted tree $T$ is defined to be the node $T(0)$.

\subsection{Vector trees}

A \textit{vector tree} $\bfct$ is a nonempty finite sequence of trees having common height; this common height is defined to be the
\textit{height} of $\bfct$ and will be denoted by $h(\bfct)$. We notice that, throughout the paper, we will start the enumeration of
vector trees with $1$ instead of $0$.

For every vector tree $\bfct=(T_1, ...,T_d)$ and every $n\in\nn$ with $n< h(\bfct)$ we set
\begin{equation} \label{e25}
\bfct\upharpoonright n=(T_1\upharpoonright n, ..., T_d\upharpoonright n).
\end{equation}
A vector tree of this form is called a \textit{vector initial subtree} of $\bfct$. Also let
\begin{equation} \label{e26}
\bfct(n)=\big(T_1(n),...,T_d(n)\big)
\end{equation}
and
\begin{equation} \label{e27}
\otimes\bfct(n)= T_1(n)\times ...\times T_d(n).
\end{equation}
The \textit{level product of $\bfct$}, denoted by $\otimes\bfct$, is defined to be the set
\begin{equation} \label{e28}
\bigcup_{n< h(\bfct)} \otimes\bfct(n).
\end{equation}
For every $\bft=(t_1,...,t_d)\in\otimes\bfct$ we set
\begin{equation} \label{e29}
\suc_{\bfct}(\bft)=\big( \suc_{T_1}(t_1), ..., \suc_{T_d}(t_d)\big).
\end{equation}
Finally, we say that a vector tree $\bfct=(T_1,...,T_d)$ is \textit{uniquely rooted} if for every $i\in\{1,...,d\}$ the tree $T_i$
is uniquely  rooted. Notice that if $\bfct$ is uniquely rooted, then $\bfct(0)=\otimes\bfct(0)$; the element $\bfct(0)$ will be called
the \textit{root} of $\bfct$.

\subsection{Strong subtrees and vector strong subtrees}

A subtree $S$ of a uniquely rooted tree $T$ is said to be \textit{strong} provided that: (a) $S$ is uniquely rooted and balanced, (b) every
level of $S$ is a subset of some level of $T$, and (c) for every non-maximal node $s\in S$ and every $t\in\immsuc_T(s)$ there exists
a unique node $s'\in\immsuc_S(s)$ such that $t\mik s'$. The \textit{level set} of a strong subtree $S$ of $T$ is defined to be the set
\begin{equation} \label{e210}
L_T(S)=\{ m\in\nn: \text{exists } n<h(S) \text{ with } S(n)\subseteq T(m)\}.
\end{equation}

The concept of a strong subtree is naturally extended to vector trees. Specifically, a \textit{vector strong subtree} of a uniquely rooted
vector tree $\bfct=(T_1,...,T_d)$ is a vector tree $\bfcs=(S_1,...,S_d)$ such that $S_i$ is a strong subtree of $T_i$ for every $i\in\{1,...,d\}$
and $L_{T_1}(S_1)= ...= L_{T_d}(S_d)$.

\subsection{Homogeneous trees and vector homogeneous trees}

Let $b\in\nn$ with $b\meg 2$. By $b^{<\nn}$ we shall denote the set of all finite sequences having values in $\{0,...,b-1\}$. The empty sequence
is denoted by $\varnothing$ and is included in $b^{<\nn}$. We view $b^{<\nn}$ as a tree equipped with the (strict) partial order $\sqsubset$ of
end-extension. Notice that $b^{<\nn}$ is a homogeneous tree with branching number $b$. If $n\meg 1$, then $b^{<n}$ stands for the initial subtree
of $b^{<\nn}$ of height $n$. For every $t,s\in b^{<\nn}$ by $t^{\con}s$ we shall denote the concatenation of $t$ and $s$.

For technical reasons we will not work with abstract homogeneous trees but with a concrete subclass. Observe that all homogeneous trees with the
same branching number are pairwise isomorphic, and so, such a restriction will have no effect in the generality of our results.
\medskip

\noindent \textbf{Convention.} \textit{In the rest of the paper by the term ``homogeneous tree" (respectively, ``finite homogeneous tree")
we will always mean a strong subtree of $b^{<\nn}$ of infinite (respectively, finite) height for some integer $b\meg 2$. For every, possibly
finite, homogeneous tree $T$ by $b_T$ we shall denote the branching number of $T$. We follow the same conventions for vector trees. In particular,
by the term ``vector homogeneous tree" (respectively, ``finite vector homogeneous tree") we will always mean a vector strong subtree of $(b_1^{<\nn},...,b_d^{<\nn})$ of infinite (respectively, finite) height for some integers $b_1,...,b_d$ with $b_i\meg 2$ for all
$i\in\{1,...,d\}$. For every, possibly finite, vector homogeneous tree $\bfct=(T_1,...,T_d)$ we set $b_{\bfct}=(b_{T_1},...,b_{T_d})$.}
\medskip

\noindent The above convention enables us to effectively enumerate the set of immediate successors of a given non-maximal node of a,
possibly finite, homogeneous tree $T$. Specifically, for every non-maximal $t\in T$ and every $p\in\{0,...,b_T-1\}$ let
\begin{equation} \label{e211}
t^{\con_T}\!p= \immsuc_T(t)\cap \suc_{b_T^{<\nn}}(t^{\con}p)
\end{equation}
and notice that
\begin{equation} \label{e212}
\immsuc_T(t)=\big\{ t^{\con_T}\!p: p\in\{0,...,b_T-1\}\big\}.
\end{equation}

\subsection{Canonical isomorphisms and vector canonical isomorphisms}

Let $T$ and $S$ be two, possibly finite, homogeneous trees with the same branching number and the same height. The \textit{canonical isomorphism}
between $T$ and $S$ is defined to be the unique bijection $\ci:T\to S$ satisfying: (a) $\ell_T(t)= \ell_S\big(\mathrm{I}(t)\big)$ for every
$t\in T$, and (b) $\ci(t^{\con_T}\!p)=\ci(t)^{\con_S}\!p$ for every non-maximal $t\in T$ and every $p\in \{0,...,b_T-1\}$. Observe that
if $R$ is a strong subtree of $T$, then the image $\ci(R)$ of $R$ under the canonical isomorphism is a strong subtree of $S$ and satisfies
$L_T(R)=L_S\big(\ci(R)\big)$.

Respectively, let $\bfct=(T_1,...,T_d)$ and $\bfcs=(S_1,...,S_d)$ be two, possibly finite, vector homogeneous trees with $b_\bfct=b_\bfcs$
and $h(\bfct)=h(\bfcs)$. For every $i\in\{1,...,d\}$ let $\ci_i$ be the canonical isomorphism between $T_i$ and $S_i$. The \textit{vector
canonical isomorphism} between $\otimes\bfct$ and $\otimes\bfcs$ is the map $\bfci: \otimes\bfct\to\otimes\bfcs$ defined by the rule
\begin{equation} \label{e213}
\bfci\big((t_1,...,t_d)\big)=\big(\ci_1(t_1), ..., \ci_d(t_d)\big).
\end{equation}
Notice that the vector canonical isomorphism $\bfci$ is a bijection.

\subsection{Milliken's Theorem}

For every finite vector homogeneous tree $\bfct$ and every integer $1\mik k\mik h(\bfct)$ by $\strong_k(\bfct)$ we shall denote the set
of all vector strong subtrees of $\bfct$ of height $k$. We will need the following elementary fact.
\begin{fact} \label{f21}
Let $d\in\nn$ with $d\meg 1$ and $b_1,...,b_d\in\nn$ with $b_i\meg 2$ for all $i\in\{1,...,d\}$. Also let $m\in\nn$ and define
\begin{equation} \label{e214}
q(b_1,...,b_d,m)=\frac{\big( \prod_{i=1}^d b_i^{b_i} \big)^{m+1} -\big( \prod_{i=1}^d b_i \big)^{m+1}}{\prod_{i=1}^d b_i^{b_i} -\prod_{i=1}^d b_i}.
\end{equation}
If $\bfct$ is a finite vector homogeneous tree with $b_{\bfct}=(b_1,...,b_d)$ and of height at least $m+2$,
then the cardinality of the set
\begin{equation} \label{e215}
\strong_2(\bfct,m+1)=\big\{ \bfcf\in\strong_2(\bfct): \otimes\bfcf(1)\subseteq \otimes\bfct(m+1) \big\}
\end{equation}
is $q(b_1,...,b_d,m)$.
\end{fact}
The following partition result is due to K. Milliken (see \cite[Theorem 2.1]{Mi2}).
\begin{thm} \label{t22}
For every integer $d\meg 1$, every $b_1,...,b_d\in\nn$ with $b_i\meg 2$ for all $i\in\{1,...,d\}$, every pair of integers $m\meg k\meg 1$
and every integer $r\meg 2$ there exists an integer $M$ with the following property. For every finite vector homogeneous tree $\bfct$ with
$b_{\bfct}=(b_1,...,b_d)$ and of height at least $M$ and every $r$-coloring of the set $\strong_k(\bfct)$ there exists $\bfcs\in\strong_m(\bfct)$
such that the set $\strong_k(\bfcs)$ is monochromatic. The least integer $M$ with this property will be denoted by $\mil(b_1,...,b_d|m,k,r)$.
\end{thm}
The original proof of Theorem \ref{t22} was ineffective, and as such, could not provide quantitative information on the numbers
$\mil(b_1,...,b_d|m,k,r)$. An analysis of the finite version of Milliken's Theorem has been carried out recently by M. Soki\'{c}
in \cite{So} yielding explicit and reasonable upper bounds. In particular, we have the following refinement of Theorem \ref{t22}.
\begin{thm} \label{t23}
For every integer $k\meg 1$ there exists a primitive recursive function $\phi_k:\nn^3\to \nn$ belonging to the class $\mathcal{E}^{5+k}$ of
Grzegorczyk's hierarchy such that for every integer $d\meg 1$, every $b_1,...,b_d\in\nn$ with $b_i\meg 2$ for all $i\in\{1,...,d\}$, every
integer $m\meg k$ and every integer $r\meg 2$ we have
\begin{equation} \label{e216}
\mil(b_1,...,b_d|m,k,r) \mik \phi_k\Big( \prod_{i=1}^d b_i^{b_i},m,r\Big).
\end{equation}
\end{thm}
Theorem \ref{t23} was not explicitly isolated in \cite{So}. For the convenience of the reader and for completeness, we will sketch
the proof in the appendix.

We will also need a certain consequence of Theorem \ref{t22}. To state it we need, first, to introduce some notation.
Specifically, for every finite vector homogeneous tree $\bfct$ and every integer $1\mik k\mik h(\bfct)$ we set
\begin{equation} \label{e217}
\strong_k^0(\bfct)=\big\{ \bfcs\in\strong_k(\bfct) : \bfcs(0)=\bfct(0)\big\}.
\end{equation}
\begin{cor} \label{c24}
Let $d\in\nn$ with $d\meg 1$ and $b_1,...,b_d\in\nn$ with $b_i\meg 2$ for every $i\in\{1,...,d\}$. Also let $m,k,r\in\nn$
with $m\meg k\meg 1$ and $r\meg 2$. If $\bfct$ is finite vector homogeneous tree with $b_{\bfct}=(b_1,...,b_d)$ and
\begin{equation} \label{e218}
h(\bfct)\meg \mil\big(\underbrace{b_1,...,b_1}_{b_1-\mathrm{times}}, ..., \underbrace{b_d,...,b_d}_{b_d-\mathrm{times}}|m,k,r\big)+1
\end{equation}
then for every $r$-coloring of $\strong_{k+1}^0(\bfct)$ there exists $\bfcr\in\strong_{m+1}^0(\bfct)$ such that the set $\strong_{k+1}^0(\bfcr)$
is monochromatic.
\end{cor}
The reduction of Corollary \ref{c24} to Theorem \ref{t22} is standard; see, e.g., \cite{Mi1,Mi2,To2}.

\subsection{The signature of a subset of a finite homogeneous tree}

Let $T$ be a finite homogeneous tree and $D$ be a subset of $T$. Following \cite{PST}, we define the \textit{signature} of $D$ in $T$ to be the set
\begin{equation} \label{e219}
\mathrm{S}_T(D)=\big\{ L_T(S): S \text{ is a strong subtree of } T \text{ with } S\subseteq D\big\}.
\end{equation}
Also let
\begin{equation} \label{e220}
w_T(D)= \sum_{n<h(T)} \dens\big(D\cap T(n)\big).
\end{equation}
The following result is due to J. Pach, J. Solymosi and G. Tardos and relates the above defined quantities (see \cite[Lemma 3']{PST}).
\begin{lem} \label{l25}
Let $T$ be a finite homogeneous tree. Then for every $D\subseteq T$ we have
\begin{equation} \label{3221}
|\mathrm{S}_T(D)| \meg \Big(\frac{b_T}{b_T-1}\Big)^{w_T(D)}.
\end{equation}
\end{lem}
Notice that, by Lemma \ref{l25}, we have
\begin{equation} \label{e222}
\udhl(b|k,\ee)=O_{b,\ee}(k)
\end{equation}
for every integer $b\meg 2$, every integer $k\meg 1$ and every real $0<\ee\mik 1$. The implied constant in (\ref{e222}) can, of course, be
estimated efficiently using Chernoff's bound.

\subsection{Two Markov-type inequalities}

Throughout the paper we will use two elementary variants of Markov's inequality. We isolate them, below, for the convenience of the reader.
\begin{fact} \label{markov-f1}
Let $0<\ee\mik 1$, $N\in\nn$ with $N\meg 1$ and $a_1,...,a_N$ in $[0,1]$. Assume that $\mathbb{E}_i a_i \meg \ee$. Then for every $0<\ee'<\ee$
we have $|\big\{ i\in \{1,...,N\}: a_i\meg \ee'\big\}| \meg (\ee-\ee')N$.
\end{fact}
\begin{fact} \label{markov-f2}
Let $0<\ee\mik 1$, $N\in\nn$ with $N\meg 1$ and $a_1,...,a_N$ in $[0,1]$ such that $\mathbb{E}_i a_i \meg \ee$.
Also let $\delta>0$ and assume that $|\big\{ i\in \{1,...,N\}: a_i\meg \ee+\delta^2\big\}| \mik \delta^3 N$. Then
$|\big\{ i\in\{1,...,N\}: a_i\meg \ee-\delta\big\}|\meg (1-\delta)N$.
\end{fact}


\section{Level selections}

We start with the following definition.
\begin{defn} \label{d31}
Let $\bfct$ be a finite vector homogeneous tree and $W$ a homogeneous tree. We say that a map $D:\otimes\bfct\to \mathcal{P}(W)$
is a \emph{level selection} if there exists a subset $L(D)=\{l_0< ... < l_{h(\bfct)-1}\}$ of $\nn$, called the \emph{level set}
of $D$, such that for every integer $n<h(\bfct)$ and every $\bft\in\otimes \bfct(n)$ we have that $D(\bft)\subseteq W(l_n)$.

For every level selection $D:\otimes\bfct\to \mathcal{P}(W)$ the \emph{height} $h(D)$ of $D$ is defined to be the height $h(\bfct)$
of the finite vector homogeneous tree $\bfct$. The \emph{density} $\delta(D)$ of $D$ is the quantity defined by
\begin{equation} \label{e31}
\delta(D)= \min\big\{\dens\big(D(\bft)\big):\bft\in\otimes\bfct\big\}.
\end{equation}
Finally, if $\bfcs$ is a vector strong subtree of $\bfct$, then by $D\upharpoonright\bfcs$ we shall denote the restriction of the
level selection $D$ on $\otimes\bfcs$.
\end{defn}
We are ready to state our main result concerning the structure of level selections.
\begin{thm} \label{t32}
For every integer $d\meg 1$, every $b_1,...,b_d,b_{d+1}\in\nn$ with $b_i\meg 2$ for all $i\in\{1,...,d+1\}$, every integer $k\meg 1$
and every real $0<\ee\mik 1$ there exists an integer $N$ with the following property. If $\bfct=(T_1,...,T_d)$ is a finite vector homogeneous
tree with $b_{\bfct}=(b_1,...,b_d)$, $W$ is a homogeneous tree with $b_W=b_{d+1}$ and $D:\otimes\bfct\to\mathcal{P}(W)$ is a level
selection with $\delta(D)\meg\ee$ and of height at least $N$, then there exist a vector strong subtree $\bfcs$ of $\bfct$
and a strong subtree $R$ of $W$ with $h(\bfcs)=h(R)=k$ and such that for every $n\in\{0,...,k-1\}$ we have
\begin{equation} \label{e32}
R(n)\subseteq \bigcap_{\bfs\in\otimes\bfcs(n)} D(\bfs).
\end{equation}
The least integer $N$ with this property will be denoted by $\ls(b_1,...,b_{d+1}|k,\ee)$.
\end{thm}
Theorem \ref{t32} is the main ingredient of the proof of Theorem \ref{t13}. Its proof will occupy the bulk of this paper and will be given
in \S 9. Let us mention, however, at this point the following simple fact which provides the link between the ``uniform density
Halpern--L\"{a}uchli" numbers and the ``level selection" numbers.
\begin{fact} \label{f33}
For every integer $d\meg 1$, every $b_1,..., b_d, b_{d+1}\in\nn$ with $b_i\meg 2$ for all $i\in\{1,...,d+1\}$, every integer $k\meg 1$
and every real $0<\ee\mik 1$ we have
\begin{equation} \label{e33}
\udhl(b_1,...,b_{d+1}|k,\ee)\mik \udhl\big(b_1,...,b_d | \ls(b_1,...,b_{d+1}|k,\ee/2),\ee/2\big).
\end{equation}
\end{fact}
\begin{proof}
For notational convenience we set $m=\ls(b_1,...,b_{d+1}|k,\ee/2)$. We fix a vector homogeneous tree $(T_1,...,T_d,W)$ with $b_{T_i}=b_i$
for all $i\in\{1,...,d\}$ and $b_W=b_{d+1}$ and we set $\bfct=(T_1,...,T_d)$. Also let $L$ be a finite subset of $\nn$ with
\begin{equation} \label{e34}
|L| \meg \udhl(b_1,...,b_d|m,\ee/2)
\end{equation}
and $D$ be a subset of the level product of $(T_1,...,T_d,W)$ such that
\begin{equation} \label{e35}
|D \cap \big(T_1(n)\times ...\times T_d(n)\times W(n)\big)| \meg  \ee |T_1(n)\times ...\times T_d(n)\times W(n)|
\end{equation}
for every $n\in L$. We need to find a vector strong subtree of $(T_1,...,T_d,W)$ of height $k$ whose level product is contained in $D$.

To this end we argue as follows. For every $n\in L$ we define a subset $C_n$ of $\otimes\bfct(n)$ by the rule
\begin{equation} \label{e36}
\bft\in C_n \Leftrightarrow \dens\big(\{w\in W(n): (\bft,w)\in D\}\big)\meg \ee/2
\end{equation}
and we observe that
\begin{equation} \label{e37}
|C_n|\meg (\ee/2)\, |\!\otimes\bfct(n)|.
\end{equation}
By (\ref{e34}) and (\ref{e37}), there exists a vector strong subtree $\bfcz$ of $\bfct$ with $h(\bfcz)=m$ and such that
$\otimes\bfcz\subseteq \bigcup_{n\in L} C_n$. Therefore, the ``section map" $D:\otimes\bfcz\to \mathcal{P}(W)$, defined by
$D(\bfz)=\{w\in W: (\bfz,w)\in D\}$ for every $\bfz\in\otimes\bfcz$, is a level selection of height $m$ and with $\delta(D)\meg \ee/2$.
By the choice of $m$, it is possible to find a vector strong subtree $\bfcs=(S_1,...,S_d)$ of $\bfcz$ and a strong
subtree $R$ of $W$ with $h(\bfcs)=h(R)=k$ and satisfying the inclusion in (\ref{e32}) for every $n\in\{0,...,k-1\}$. It follows that
$(S_1,...,S_d,R)$ is a vector strong subtree of $(T_1,...,T_d,W)$ of height $k$ whose level product is contained in $D$.
Thus, the proof is completed.
\end{proof}


\section{Outline of the argument}

In this section we will give a detailed outline of the proof of Theorem \ref{t13}. The first step is given in Fact \ref{f33}. Indeed,
by Fact \ref{f33}, the task of estimating the ``uniform density Halpern--L\"{a}uchli" numbers reduces to that of estimating the
``level selection" numbers. To achieve this goal, we will follow an inductive procedure which can be schematically described as follows:
\begin{equation} \label{e41}
\left. \begin{array} {ll} \udhl(b_1,...,b_d|\ell,\eta) & \text{for every $\ell$ and $\eta$} \\
\ls(b_1,...,b_{d+1}|k,\eta) & \text{for every $\eta$} \end{array}  \right\} \Rightarrow \ls(b_1,...,b_{d+1}|k+1,\ee).
\end{equation}
Precisely, in order to estimate the number $\ls(b_1,...,b_{d+1}|k+1,\ee)$ we need to have at our disposal the numbers
$\udhl(b_1,...,b_d|\ell,\eta)$ as well as the numbers $\ls(b_1,...,b_{d+1}|k,\eta)$ for every integer $\ell\mik\ell_0$ and every
$\eta\in [\theta_0, 1]$ where $\ell_0$ is a large enough integer and $\theta_0$ is an appropriately chosen positive constant
which is very small compared with the given density $\ee$.

Before we proceed to discuss the main arguments of the proof we need to make some important observations. Specifically, let $\bfct$
be a finite vector homogeneous tree and suppose that we are given a subset $A$ of the level product $\otimes\bfct$ of $\bfct$. We need
an effective way to measure the size of the set $A$ where the word ``effective" should be interpreted as ``taking into account how the
set $A$ is distributed along the products of different levels of $\bfct$". Notice that the uniform probability measure on $\otimes\bfct$
is rather ineffective in this regard since it is highly concentrated on the products of very few of the last levels of $\bfct$. There is,
however, a very natural way to overcome this problem, discovered by H. Furstenberg and B. Weiss in \cite{FW}. Specifically, let
\begin{equation} \label{e42}
\mu_{\bfct}(A)=\ave_{n< h(\bfct)} \frac{|A\cap\otimes\bfct(n)|}{|\!\otimes\bfct(n)|}.
\end{equation}
Actually, H. Furstenberg and B. Weiss considered finite homogeneous trees instead of level products of finite vector homogeneous trees.
It is completely straightforward, however, to extend their definition to the higher-dimensional case. The reader should have in mind that,
in what follows, when we say that a certain property holds for ``many" or for ``almost every" $\bft\in\otimes\bfct$, then we will refer
to the probability measure $\mu_{\bfct}$.

After this preliminary discussion we are ready to comment on the proof of the basic step of the inductive scheme described in (\ref{e41}).
So assume that the parameters $b_1,...,b_{d+1}, k$ and $\ee$ are fixed and that we are given: (i) a finite vector homogeneous
tree $\bfct$ with $b_{\bfct}=(b_1,...,b_d)$, (ii) a homogeneous tree $W$ with $b_W=b_{d+1}$, and (iii) a level selection
$D:\otimes\bfct\to\mathcal{P}(W)$ with $\delta(D)\meg \ee$ and of sufficiently large height. What we need to find is a
vector strong subtree $\bfcs$ of $\bfct$ and a strong subtree $R$ of $W$ with $h(\bfcs)=h(R)=k+1$ and such that
\begin{equation} \label{e43}
R(n)\subseteq \bigcap_{\bfs\in\otimes\bfcs(n)} D(\bfs)
\end{equation}
for every $n\in\{0,...,k\}$.

Let $r$ be a small enough parameter depending on our data $b_1,...,b_{d+1}, k$ and $\ee$. The first observation we make is quite standard
in proofs of this sort: we can assume that for every $w\in W$ we cannot increase the density of the level selection $D$ to $\ee+r$
by restricting its values to the subtree $\suc_W(w)$. Indeed, suppose that there exists a node $w\in W$ such that for ``many"
$\bft\in\otimes\bfct$ the density of the set $D(\bft)$ relative to the node $w$ is at least $\ee+r$. Then we can find a new level
selection $D'$ whose graph is contained in $D$ and such that $\delta(D')\meg \ee+r$. The number of times this can happen is, of course,
bounded by $\lceil 1/r\rceil$ until the density reaches $1$ and we can finish the proof in a particularly simple way.

Thus, in what follows we can assume that we have ``lack of density increment", or equivalently, that the following
concentration hypothesis holds true.
\begin{enumerate}
\item[(H)] If $\bfcz$ is a vector strong subtree of $\bfct$ of sufficiently large height, then for ``almost every" $w\in W$
and for ``almost every" $\bfz\in\otimes\bfcz$ the density of the set $D(\bfz)$ relative to the node $w$ is roughly $\ee$.
\end{enumerate}
With this information at hand, we devise an algorithm in order to find the desired trees $\bfcs$ and $R$. The number of times
we need to iterate this algorithm is at most $K_0=\udhl\big(b_1,...,b_d|2,\ee/(4b_{d+1})\big)$, and so, it is \textit{a priori} controlled.
Each time we perform the following three basic steps.

\subsection*{Step 1}

Suppose that we are at stage $n+1$. From the previous iteration we will have as an input a vector strong subtree $\bfcz_n$ of $\bfct$,
a node $w_n\in W$ and $p_n\in\{0,...,b_{d+1}-1\}$ satisfying certain properties. These properties are used to define a subset $A_{n+1}$
of $\suc_W(w_n^{\con_W}\!p_n)$ with $\dens(A_n \ | \ w_n^{\con_W}\!p_n)\meg \ee/8$. We view the set $A_{n+1}$ as an ``admissible" subset
of $W$. Next, we find a vector strong subtree $\bfcv$ of $\bfcz_n$ with $\bfcv\upharpoonright n+1= \bfcz_n\upharpoonright n+1$ and of
sufficiently large height, as well as, a node $w\in A_{n+1}$ such that for every integer $m>n+1$ and every $\bfv\in \otimes\bfcv(m)$
the density of the set $D(\bfv)$ relative to \textit{every} immediate successor $w'$ of $w$ is almost $\ee$. This is achieved using
hypothesis (H).

\subsection*{Step 2}

We perform coloring arguments, using Milliken's Theorem, in order to find the desired trees $\bfcs$ and $R$. If we do not succeed,
then we will be able to select a vector strong subtree $\bfcb$ of $\bfcv$ with $\bfcb\upharpoonright n+1=\bfcv\upharpoonright n+1$
and of sufficiently large height, a subset $\Gamma_{n+1}$ of $\otimes\bfcv(n+1)$ of cardinality at least $(\ee/4b_{d+1})|\!\otimes\bfcv(n+1)|$
and $p\in\{0,...,b_{d+1}-1\}$ with the following property. For every $\bfv\in \Gamma_{n+1}$ and every $\bfcs\in \strong_{k+1}(\bfcb)$ with
$\bfcs(0)=\bfv$ the set
\begin{equation} \label{e44}
\bigcup_{n=1}^k \Big( \bigcap_{\bfs\in\otimes\bfcs(n)} D(\bfs)\cap \suc_W(w^{\con_W}\!p) \Big)
\end{equation}
does not contain a strong subtree of $W$ of height $k$.

\subsection*{Step 3}

We use the aforementioned property to show that the level selection $D$ is rather ``thin" when restricted to $\suc_W(w^{\con_W}\!p)$.
Specifically, let $\Gamma_0,..., \Gamma_{n+1}$ be the sets obtained by Step 2 from all previous iterations. We find a vector strong subtree
$\bfcz'$ of $\bfcb$ with $\bfcz'\upharpoonright n+1=\bfcb\upharpoonright n+1$ and of sufficiently large height and satisfying the following.
For every $\bfcf\in\strong_2(\bfcz')$ with $\bfcf(0)\in \Gamma_0\cup ...\cup \Gamma_{n+1}$ and $\otimes\bfcf(1)\subseteq \otimes\bfcz'(n+2)$
the density of the set
\begin{equation} \label{e45}
\bigcap_{\bft\in\otimes\bfcf(1)} D(\bft)
\end{equation}
relative to the node $w^{\con_W}\!p$ is essentially negligible. We set $\bfcz_{n+1}=\bfcz'$, $w_{n+1}=w$ and $p_{n+1}=p$ and
we go back to Step 1.
\medskip

If after $K_0$ iterations the desired trees $\bfcs$ and $R$ have not been found, then using the sets $\{\Gamma_0,...,\Gamma_{K_0-1}\}$
obtained by Step 2 we can easily derive a contradiction. Having shown that the above algorithm does locate the trees $\bfcs$ and $R$,
we can analyze each step separately and estimate the number $\ls(b_1,...,b_{d+1}|k+1,\ee)$. And as we have already pointed out, this is
enough to complete the proof of Theorem \ref{t13}.


\section{Step 1: obtaining strong denseness}

We start with the following definition. It is a crucial conceptual step towards the proof of Theorem \ref{t13}.
\begin{defn} \label{d61}
Let $\bfct$ be a finite vector homogeneous tree, $W$ a homogeneous tree and $D:\otimes\bfct\to \mathcal{P}(W)$ a level selection.
Also let $\bfcs$ be a vector strong subtree of $\bfct$, $w\in W$ and $0<\ee\mik 1$.
\begin{enumerate}
\item[(1)] We say that $D$ is \emph{$(w,\bfcs,\ee)$-dense} provided that $\ell_W(w)\mik \min L(D\upharpoonright\bfcs)$
and $\dens(D(\bfs) \ | \ w)\meg \ee$ for every $\bfs\in\otimes\bfcs$.
\item[(2)] We say that $D$ is \emph{$(w,\bfcs,\ee)$-strongly dense} if $D$ is $(w',\bfcs,\ee)$-dense for every $w'\in\immsuc_W(w)$.
\end{enumerate}
\end{defn}

\subsection{Lack of density increment implies strong denseness}

For every $0<\alpha\mik \beta\mik 1$ and every $0<\varrho\mik 1$ we set
\begin{equation} \label{e61}
\gamma_0=\gamma_0(\alpha,\beta,\varrho)=(\beta+\varrho^2-\alpha)^{1/2},
\end{equation}
\begin{equation} \label{e62}
\gamma_1=\gamma_1(\alpha,\beta,\varrho)=(\gamma_0+\gamma_0^2)^{1/2}
\end{equation}
and
\begin{equation} \label{e63}
\gamma_2=\gamma_2(\alpha,\beta,\varrho)=(\gamma_1+\gamma_1^2)^{1/2}.
\end{equation}
The following lemma corresponds to the first step of the proof of Theorem \ref{t13}. It is based on the phenomenon we described in \S 4,
namely that ``lack of density increment" implies a strong concentration hypothesis.
\begin{lem} \label{l62}
Let $d\in\nn$ with $d\meg 1$ and $b_1,...,b_d\in\nn$ with $b_i\meg 2$ for every $i\in\{1,...,d\}$ and assume that
for every integer $n\meg 1$ and every $0<\eta\mik 1$ the numbers $\udhl(b_1,...,b_d|n,\eta)$ have been defined.

Also let $0<\alpha\mik\beta\mik 1$, $0<\varrho\mik 1$ and $b,q\in\nn$ with $b\meg 2$ and $q\meg 1$ such that
\begin{equation} \label{e64}
\gamma_0\mik \Big( \frac{\alpha}{4qb}\Big)^4.
\end{equation}
Assume that we are given
\begin{enumerate}
\item[(a)] a finite vector homogeneous tree $\bfcu$ with $b_{\bfcu}=(b_1,...,b_d)$,
\item[(b)] a homogeneous tree $W$ with $b_W=b$,
\item[(c)] a node $w_0\in W$ and $\ell\in\nn$ with $\ell_W(w_0)\mik \ell$,
\item[(d)] a subset $A$ of $\suc_W(w_0)\cap W(\ell)$ with $\dens(A \ | \ w_0)\meg \beta/8$,
\item[(e)] a subset $L$ of $\nn$ with $|L|=h(\bfcu)$ and $\ell<\min L$, and
\item[(f)] for every $j\in\{1,...,q\}$ a level selection $D_j:\otimes\bfcu\to\mathcal{P}(W)$ with $L(D_j)=L$
and which is $(w_0,\bfcu,\alpha)$-dense.
\end{enumerate}
Finally, let $N\in\nn$ with $N\meg 1$ and suppose that
\begin{equation} \label{e65}
h(\bfcu)\meg \frac{1}{\varrho^3} \udhl(b_1,...,b_d|N,\varrho^3).
\end{equation}
Then, either
\begin{enumerate}
\item[(i)] there exist a vector strong subtree $\bfcu'$ of $\bfcu$ with $h(\bfcu')=N$, $j_0\in\{1,...,q\}$ and a node
$w_0'\in\suc_W(w_0)\cap W(\ell+1)$ such that the level selection $D_{j_0}$ is $(w_0',\bfcu',\beta+\varrho^2/2)$-dense, or
\item[(ii)] there exist a vector strong subtree $\bfcu''$ of $\bfcu$ with $h(\bfcu'')=N$ and a node $w_0''\in A$ such that
$D_j$ is $(w_0'',\bfcu'',\alpha-\gamma_0-\gamma_1-\gamma_2)$-strongly dense for every $j\in\{1,...,q\}$.
\end{enumerate}
\end{lem}

\subsection{Proof of Lemma \ref{l62}}

We will consider four cases. The first three cases imply that alternative (i) holds true while the last one yields alternative (ii).
Before we proceed to the details, we isolate for future use the following elementary facts.
\begin{enumerate}
\item[($\mathcal{P}$1)] $\alpha+\gamma_0^2=\alpha-\gamma_0+\gamma_1^2=\alpha-\gamma_0-\gamma_1+\gamma_2^2=\beta+\varrho^2$.
\item[($\mathcal{P}$2)] $0<\varrho\mik \gamma_0 < \beta/(8bq)<1/8$.
\item[($\mathcal{P}$3)] $\gamma_1\mik 2 \gamma_0^{1/2}$ and $\gamma_2\mik 2 \gamma_0^{1/4}$.
\end{enumerate}
The above properties are straightforward consequences of (\ref{e61}), (\ref{e62}), (\ref{e63}) and (\ref{e64}).
Also for every $j\in\{1,...,q\}$ and every $w\in\suc_W(w_0)\cap W(\ell+1)$ we set
\begin{equation} \label{e66}
\Delta_{j,w}= \big\{ \bfu\in\otimes\bfcu: \dens(D_j(\bfu) \ | \ w)\meg \beta+\varrho^2/2\big\},
\end{equation}
\begin{equation} \label{e67}
I_{j,w}=\big\{ n<h(\bfcu): \ave_{\bfu\in\otimes\bfcu(n)} \dens(D_j(\bfu) \ | \ w)\meg \beta+\varrho^2\big\}
\end{equation}
and
\begin{equation} \label{e68}
K_{j,w}=\big\{ n<h(\bfcu): |\Delta_{j,w}\cap \otimes\bfcu(n)|\meg \varrho^3 |\!\otimes\bfcu(n)|\big\}.
\end{equation}
After this preliminary discussion, we are ready to distinguish cases.
\medskip

\noindent \textsc{Case 1:} \textit{there exist $j_0\in\{1,...,q\}$ and a node $w_0'\in\suc_W(w_0)\cap W(\ell+1)$ such that
$\ave_{n<h(\bfcu)}\ave_{\bfu\in\otimes\bfcu(n)} \dens(D_{j_0}(\bfu) \ | \ w_0')\meg\beta+\varrho^2$}. Using our hypotheses and applying
Fact \ref{markov-f1} twice, we see that there exists $\Lambda\subseteq \{0,...,h(\bfcu)-1\}$ with $|\Lambda|\meg (\varrho^2/4) h(\bfcu)$
and such that $|\Delta_{j_0,w_0'}\cap\otimes\bfcu(n)|\meg (\varrho^2/4)|\!\otimes\bfcu(n)|$ for every $n\in\Lambda$. Notice that
\begin{equation} \label{e69}
|\Lambda| \meg \frac{\varrho^2}{4} h(\bfcu) \stackrel{(\ref{e65})}{\meg} \frac{1}{4\varrho}\udhl(b_1,...,b_d|N,\varrho^3)
\stackrel{(\mathcal{P}2)}{\meg} \udhl(b_1,...,b_d|N,\varrho^2/4).
\end{equation}
Therefore, there exists a vector strong subtree $\bfcu'$ of $\bfcu$ with $h(\bfcu')=N$ and such that $\otimes\bfcu'\subseteq \Delta_{j_0,w_0'}$.
Hence the level selection $D_{j_0}$ is $(w_0',\bfcu',\beta+\varrho^2/2)$-dense, and so this case implies part (i) of the lemma.
\medskip

\noindent \textsc{Case 2:} \textit{there exist $j_0\in\{1,...,q\}$ and a node $w_0'\in\suc_W(w_0)\cap W(\ell+1)$ such that
$|I_{j_0,w_0'}|\meg \varrho^3 h(\bfcu)$}. By Fact \ref{markov-f1}, we have $|\Delta_{j_0,w_0'}\cap\otimes\bfcu(n)|\meg (\varrho^2/2)|\!\otimes\bfcu(n)|$ for every $n\in I_{j_0,w_0'}$. Moreover,
\begin{equation} \label{e610}
|I_{j_0,w_0'}| \meg \varrho^3 h(\bfcu) \stackrel{(\ref{e65})}{\meg} \udhl(b_1,...,b_d|N,\varrho^3)
\stackrel{(\mathcal{P}2)}{\meg} \udhl(b_1,...,b_d|N,\varrho^2/2).
\end{equation}
Arguing as above, we see that this case also implies part (i) of the lemma.
\medskip

\noindent \textsc{Case 3:} \textit{there exist $j_0\in\{1,...,q\}$ and a node $w_0'\in\suc_W(w_0)\cap W(\ell+1)$ such that
$|K_{j_0,w_0'}|\meg \varrho^3 h(\bfcu)$}. By (\ref{e65}), we have $|K_{j_0,w_0'}|\meg \udhl(b_1,...,b_d|N,\varrho^3)$. Moreover,
by the definition of the set $K_{j_0,w_0'}$ in (\ref{e68}), we see that $|\Delta_{j_0,w_0'}\cap\otimes\bfcu(n)|\meg \varrho^3
|\!\otimes\bfcu(n)|$ for every $n\in K_{j_0,w_0'}$. Thus, this case also implies part (i) of the lemma.
\medskip

\noindent \textsc{Case 4:} \textit{none of the above cases holds true}. In this case we will show that the second alternative of the lemma
is satisfied. It is useful at this point to isolate which hypotheses we have at our disposal. In particular, notice that for every
$j\in\{1,...,q\}$ and every $w\in\suc_W(w_0)\cap W(\ell+1)$ we have
\begin{enumerate}
\item[(H1)] $\ave_{n<h(\bfcu)}\ave_{\bfu\in\otimes\bfcu(n)} \dens(D_j(\bfu) \ | \ w) <\beta+\varrho^2$,
\item[(H2)] $|I_{j,w}| < \varrho^3 h(\bfcu)$ and
\item[(H3)] $|K_{j,w}| < \varrho^3 h(\bfcu)$.
\end{enumerate}
We set
\begin{equation} \label{e611}
\alpha_0=\alpha-\gamma_0
\end{equation}
and we define $B\subseteq \suc_W(w_0)\cap W(\ell+1)$ by the rule
\begin{equation} \label{e612}
w\in B \Leftrightarrow \ave_{n<h(\bfcu)}\ave_{\bfu\in\otimes\bfcu(n)} \dens(D_j(\bfu) \ | \ w)\meg \alpha_0
\text{ for every } j\in\{1,...,q\}.
\end{equation}
\begin{claim} \label{c63}
We have $|B|\meg (1-q\gamma_0)|\suc_W(w_0)\cap W(\ell+1)|$.
\end{claim}
\begin{proof}[Proof of Claim \ref{c63}]
For every $j\in\{1,...,q\}$ and every $w\in\suc_W(w_0)\cap W(\ell+1)$ we set
$\epsilon_{j,w}=\ave_{n<h(\bfcu)}\ave_{\bfu\in\otimes\bfcu(n)} \dens(D_j(\bfu) \ | \ w)$. Also let
\begin{equation} \label{e613}
B_j=\{ w\in\suc_W(w_0)\cap W(\ell+1): \epsilon_{j,w}\meg\alpha_0\}.
\end{equation}
Clearly, it is enough to show that $|B_j|\meg (1-\gamma_0)|\suc_W(w_0)\cap W(\ell+1)|$ for every $j\in\{1,...,q\}$.
To this end let $j\in\{1,...,q\}$ be arbitrary. By (H1), for every node $w\in\suc_W(w_0)\cap W(\ell+1)$ we have
\begin{equation} \label{e614}
\epsilon_{j,w} <\beta+\varrho^2 \stackrel{(\mathcal{P}1)}{=} \alpha+\gamma_0^2.
\end{equation}
On the other hand, the level selection $D_j:\otimes\bfcu\to\mathcal{P}(W)$ is $(w_0,\bfcu,\alpha)$-dense.
Since the tree $W$ is homogeneous, this yields that
\begin{equation} \label{e615}
\ave_{w\in\suc_W(w_0)\cap W(\ell+1)} \epsilon_{j,w} =  \ave_{n<h(\bfcu)} \ave_{\bfu\in\otimes\bfcu(n)} \dens(D_j(\bfu) \ | \ w_0) \meg \alpha.
\end{equation}
Combining (\ref{e614}) and (\ref{e615}) and using Fact \ref{markov-f2}, the result follows.
\end{proof}
\begin{claim} \label{c64}
There exists $w_0''\in A$ such that $\immsuc_W(w_0'')\subseteq B$.
\end{claim}
\begin{proof}[Proof of Claim \ref{c64}]
We set $F=\{ w\in \suc_W(w_0)\cap W(\ell): \immsuc_W(w)\nsubseteq B\}$. By Claim \ref{c63} and using the fact that the branching
number of the homogeneous tree $W$ is $b$, we see that the cardinality of $F$ is at most $bq\gamma_0 |\suc_W(w_0)\cap W(\ell)|$.
On the other hand, we have $|A|\meg (\beta/8) |\suc_W(w_0)\cap W(\ell)|$. Invoking property ($\mathcal{P}$2), the result follows.
\end{proof}
Now we set
\begin{equation} \label{e616}
\alpha_1=\alpha_0-\gamma_1 \stackrel{(\ref{e611})}{=}\alpha-\gamma_0-\gamma_1
\end{equation}
and we define $\mathcal{N}\subseteq \{0,...,h(\bfcu)-1\}$ by the rule
\begin{eqnarray} \label{e617}
n\in \mathcal{N} & \Leftrightarrow & \ave_{\bfu\in\otimes\bfcu(n)} \dens(D_j(\bfu) \ | \ w)\meg \alpha_1
\text{ for every } j\in\{1,...,q\} \\
&  & \text{and every } w\in\immsuc_W(w_0''). \nonumber
\end{eqnarray}
\begin{claim} \label{c65}
We have $|\mathcal{N}|\meg (1-bq\gamma_1) h(\bfcu)$.
\end{claim}
\begin{proof}[Proof of Claim \ref{c65}]
We will argue as in the proof of  Claim \ref{c63}. Specifically, for every $j\in\{1,...,q\}$, every $w\in\immsuc_W(w_0'')$ and every
integer $n<h(\bfcu)$ we set $\epsilon_{j,w,n}=\ave_{\bfu\in\otimes\bfcu(n)} \dens(D_j(\bfu) \ | \ w)$ and
$\mathcal{N}_{j,w}=\{n<h(\bfcu): \epsilon_{j,w,n}\meg \alpha_1\}$. Since the branching number of the homogeneous tree $W$ is $b$,
it is enough to show that $|\mathcal{N}_{j,w}|\meg (1-\gamma_1) h(\bfcu)$ for every $j\in\{1,...,q\}$ and every $w\in\immsuc_W(w_0'')$.
So fix $j\in\{1,...,q\}$ and $w\in\immsuc_W(w_0'')$. By (H2), we see that
\begin{equation} \label{e618}
|\{n<h(\bfcu): \epsilon_{j,w,n}\meg \beta+\varrho^2\}|< \varrho^3 h(\bfcu) \stackrel{(\mathcal{P}2),(\ref{e62})}{\mik} \gamma_1^3 h(\bfcu).
\end{equation}
By Claim \ref{c64}, we have $w\in\immsuc_W(w_0'')\subseteq B$. Therefore, invoking the definition of the set $B$ given
in (\ref{e612}), we get
\begin{equation} \label{e619}
\ave_{n<h(\bfcu)} \epsilon_{j,w,n} \meg \alpha_0.
\end{equation}
Finally, by ($\mathcal{P}$1) and (\ref{e611}), we have $\beta+\varrho^2=\alpha_0+\gamma_1^2$. Thus, combining the estimates in
(\ref{e618}) and (\ref{e619}) and using Fact \ref{markov-f2}, the result follows.
\end{proof}
Next we set
\begin{equation} \label{e620}
\alpha_2=\alpha_1-\gamma_2 \stackrel{(\ref{e616})}{=} \alpha-\gamma_0-\gamma_1-\gamma_2
\end{equation}
and
\begin{equation} \label{e621}
\mathcal{N}^* =\mathcal{N} \setminus \Big(\bigcup_{j=1}^q \bigcup_{w\in\immsuc_W(w_0'')} K_{j,w}\Big).
\end{equation}
Notice that if $n\in\mathcal{N}^*$, then for every $j\in\{1,...,q\}$ and every $w\in\immsuc_W(w_0'')$ we have that
$n\notin K_{j,w}$ and so
\begin{eqnarray} \label{e622}
|\{\bfu\in\otimes\bfcu(n) : \dens(D_j(\bfu) \ | \ w)\meg\beta+\varrho^2\}| & \stackrel{(\ref{e66})}{\mik} & |\Delta_{j,w}\cap\otimes\bfcu(n)| \\
& \stackrel{(\ref{e68})}{\mik} & \varrho^3 |\!\otimes\bfcu(n)| \nonumber \\
& \stackrel{(\mathcal{P}2),(\ref{e63})}{\mik} & \gamma_2^3 |\!\otimes\bfcu(n)|. \nonumber
\end{eqnarray}
\begin{claim} \label{c66}
The following hold.
\begin{enumerate}
\item[(i)] We have $|\mathcal{N}^*|\meg (1-bq\gamma_1-bq\varrho^3) h(\bfcu)$.
\item[(ii)] For every $n\in\mathcal{N}^*$, every $j\in\{1,...,q\}$ and every $w\in\immsuc_W(w_0'')$ we have that
$|\{\bfu\in\otimes\bfcu(n):\dens(D_j(\bfu) \ | \ w)\meg\alpha_2\}|\meg (1-\gamma_2)|\!\otimes\bfcu(n)|$.
\end{enumerate}
\end{claim}
\begin{proof}[Proof of Claim \ref{c66}]
By Claim \ref{c64} and our assumptions for the set $A$, we see that $\immsuc_W(w_0'')\subseteq \suc_W(w_0)\cap W(\ell+1)$.
Thus, part (i) follows by Claim \ref{c65} and hypothesis (H3). On the other hand, by property ($\mathcal{P}$1) and (\ref{e616}),
we have that $\beta+\varrho^2=\alpha_1+\gamma_2^2$. Therefore, part (ii) follows by (\ref{e620}), (\ref{e622}) and Fact \ref{markov-f2}.
\end{proof}
We are ready for the final step of the argument. Let $\Delta^*$ be the subset of $\otimes\bfcu$ defined by the rule
\begin{eqnarray} \label{e623}
\bfu\in\Delta^* & \Leftrightarrow & \dens(D_j(\bfu) \ | \ w)\meg \alpha_2 \text{ for every } w\in \immsuc_W(w_0'') \\
& & \text{and every }  j\in\{1,...,q\}. \nonumber
\end{eqnarray}
By part (ii) of Claim \ref{c66}, we get that $|\Delta^*\cap\otimes\bfcu(n)|\meg (1-bq\gamma_2)|\!\otimes\bfcu(n)|$
for every $n\in \mathcal{N}^*$. Moreover,
\begin{equation} \label{e624}
1-bq\gamma_2 \stackrel{(\mathcal{P}3)}{\meg} 1- 2bq \gamma_0^{1/4}
\stackrel{(\ref{e64})}{\meg} 1-2bq \frac{\alpha}{4bq} \meg \frac{1}{2} \stackrel{(\mathcal{P}2)}{\meg} \varrho^3
\end{equation}
and, by part (i) of Claim \ref{c66},
\begin{eqnarray} \label{e625}
|\mathcal{N}^*| & \meg & (1-bq\gamma_1-bq\varrho^3) h(\bfcu) \stackrel{(\mathcal{P}2),(\ref{e62})}{\meg} (1-2bq\gamma_1) h(\bfcu) \\
& \stackrel{(\mathcal{P}3)}{\meg} & (1-4bq\gamma_0^{1/2})h(\bfcu) \stackrel{(\ref{e64})}{\meg}
\Big( 1-4bq \Big(\frac{\alpha}{4bq}\Big)^2\Big) h(\bfcu) \nonumber \\
& \meg & \frac{1}{2}h(\bfcu) \stackrel{(\mathcal{P}2)}{\meg} \varrho^3 h(\bfcu) \stackrel{(\ref{e65})}{\meg}
\udhl(b_1,...,b_d|N,\varrho^3). \nonumber
\end{eqnarray}
Therefore, there exists a vector strong subtree $\bfcu''$ of $\bfcu$ with $h(\bfcu'')=N$ and such that $\otimes\bfcu''\subseteq
\Delta^*$. Invoking the definition of $\Delta^*$ and (\ref{e620}), we conclude that for every $j\in\{1,...,q\}$ the level selection
$D_j$ is $(w_0'',\bfcu'',\alpha-\gamma_0-\gamma_1-\gamma_2)$-strongly dense. Thus, the proof of Lemma \ref{l62} is completed.

\subsection{Consequences}
Lemma \ref{l62} will be used, later on, in a rather special form. We isolate, below, the exact statement that we need.
\begin{cor} \label{c67}
Let $d\in\nn$ with $d\meg 1$ and $b_1,...,b_d\in\nn$ with $b_i\meg 2$ for every $i\in\{1,...,d\}$ and assume that
for every integer $n\meg 1$ and every $0<\eta\mik 1$ the numbers $\udhl(b_1,...,b_d|n,\eta)$ have been defined.

Let $0<\alpha\mik\beta\mik 1$ and $0<\varrho\mik 1$ and define $\gamma_0$, $\gamma_1$ and $\gamma_2$ as in (\ref{e61}), (\ref{e62})
and (\ref{e63}) respectively. Also let $b,m\in\nn$ with $b\meg 2$ and such that
\begin{equation} \label{e626}
\gamma_0\mik \Big( \frac{\alpha}{4(\prod_{i=1}^d b_i)^{m+1}b}\Big)^4.
\end{equation}
Assume that we are given
\begin{enumerate}
\item[(a)] a finite vector homogeneous tree $\bfcz$ with $b_{\bfcz}=(b_1,...,b_d)$ and $h(\bfcz)\meg m+2$,
\item[(b)] a homogeneous tree $W$ with $b_W=b$,
\item[(c)] a node $\tilde{w}\in W$ and $\ell\in\nn$ with $\ell_W(\tilde{w})\mik \ell$,
\item[(d)] a subset $A$ of $\suc_W(\tilde{w})\cap W(\ell)$ with $\dens(A \ | \ \tilde{w})\meg \beta/8$,
\item[(e)] a subset $L=\{l_0<...< l_{h(\bfcz)-1}\}$ of $\nn$ such that $l_m=\ell$ and
\item[(f)] a level selection $D:\otimes\bfcz\to\mathcal{P}(W)$ which is $(\tilde{w},\suc_{\bfcz}(\bfz),\alpha)$-dense
for every $\bfz\in\otimes\bfcz(m+1)$ and with $L(D)=L$.
\end{enumerate}
Finally, let $N\in\nn$ with $N\meg 1$ and suppose that
\begin{equation} \label{e627}
h(\bfcz)\meg (m+1)+\frac{1}{\varrho^3} \udhl(b_1,...,b_d|N,\varrho^3).
\end{equation}
Then, either
\begin{enumerate}
\item[(i)] there exist a vector strong subtree $\bfcv$ of $\bfcz$ with $h(\bfcv)=N$ and a node
$w'\in\suc_W(\tilde{w})\cap W(\ell+1)$ such that $D$ is $(w',\bfcv,\beta+\varrho^2/2)$-dense, or
\item[(ii)] there exist a vector strong subtree $\bfcz'$ of $\bfcz$ with $\bfcz'\upharpoonright m=\bfcz\upharpoonright m$ and $h(\bfcz')=(m+1)+N$
and a node $w''\in A$ such that the level selection $D$ is $(w'',\suc_{\bfcz'}(\bfz),\alpha-\gamma_0-\gamma_1-\gamma_2)$-strongly dense
for every $\bfz\in\otimes\bfcz'(m+1)$.
\end{enumerate}
\end{cor}
\begin{proof}
Let $\lambda=h(\bfcz)-(m+1)$ and $q=(\prod_{i=1}^d b_i)^{m+1}$. Notice that $|\!\otimes\bfcz(m+1)|=q$. Also we write $\bfcz=(Z_1,...,Z_d)$
and we set $\bfcb=(b_1^{<\lambda},...,b_d^{<\lambda})$.

Let $\bfz\in\otimes \bfcz(m+1)$ be arbitrary and write $\suc_{\bfcz}(\bfz)=(Z_1^\bfz,...,Z_d^\bfz)$. For every $i\in\{1,...,d\}$ the finite
homogeneous trees $b_i^{<\lambda}$ and $Z_i^\bfz$ have the same branching number and the same height. Therefore, as we described in \S 2.5,
we may consider the canonical isomorphism $\ci_i^\bfz:b_i^{<\lambda}\to Z_i^\bfz$. The same remarks, of course, apply to the finite vector
homogeneous trees $\bfcb$ and $\suc_{\bfcz}(\bfz)$. Thus we may also consider the vector canonical isomorphism $\bfci_\bfz:\otimes\bfcb\to\otimes\suc_{\bfcz}(\bfz)$ given by the family of maps $\big\{\ci_i^\bfz:i\in\{1,...,d\}\big\}$ via formula (\ref{e213}).
Notice that for every $\bfz,\bft\in\otimes\bfcz(m+1)$ and every $i\in\{1,...,d\}$ if $Z_i^\bfz=Z_i^\bft$ (that is, if the finite sequences $\bfz$
and $\bft$ agree on the $i$-th coordinate), then the maps $\ci_i^\bfz$ and $\ci_i^\bft$ are identical.

For every $\bfz\in\otimes \bfcz(m+1)$ we define a level selection $D_\bfz:\otimes\bfcb\to\mathcal{P}(W)$ by
\begin{equation} \label{e628}
D_\bfz(\bfu)= D\big( \bfci_\bfz(\bfu)\big).
\end{equation}
It is then clear that we may apply Lemma \ref{l62} to the family $\{D_\bfz:\bfz\in\otimes\bfcz(m+1)\}$.

If the first alternative of the lemma holds true, then we get a vector strong subtree $\bfcu'$ of $\bfcb$ of height $N$,
$\bfz_0\in\otimes\bfcz(m+1)$ and a node $w'\in\suc_W(\tilde{w})\cap W(\ell+1)$ such that the level selection $D_{\bfz_0}$ is
$(w',\bfcu',\beta+\varrho^2/2)$-dense. We set $\bfcv=\bfci_{\bfz_0}(\bfcu')$ and we observe that with this choice
part (i) of the corollary is satisfied.

Otherwise, we get a vector strong subtree $\bfcu''=(U_1,...,U_d)$ of $\bfcb$ of height $N$ and a node $w''\in A$
such that $D_\bfz$ is $(w'',\bfcu'',\alpha-\gamma_0-\gamma_1-\gamma_2)$-strongly dense for every $\bfz\in\otimes\bfcz(m+1)$.
In this case, for every $i\in\{1,...,d\}$ let
\begin{equation} \label{e629}
Z_i'= (Z_i\upharpoonright m) \cup \big\{ \ci_i^\bfz(U_i) : \bfz\in\otimes\bfcz(m+1)\big\}
\end{equation}
and set $\bfcz'=(Z_1',...,Z_d')$. It is easy to check that with this choice part (ii) of the corollary is satisfied.
The proof is completed.
\end{proof}


\section{Step 2: obtaining the set $\Gamma$ and fixing the ``direction"}

Our goal in this section is to analyze the second step of the proof of Theorem \ref{t13}.
This is, essentially, the content of the following lemma.
\begin{lem} \label{l71}
Let $d\in\nn$ with $d\meg 1$ and $b_1,...,b_d\in\nn$ with $b_i\meg 2$ for every $i\in\{1,...,d\}$. Also let $b,q,k\in\nn$
with $b\meg 2$ and $q,k\meg 1$.

Assume that we are given
\begin{enumerate}
\item[(a)] a finite vector homogeneous tree $\bfcu$ with $b_{\bfcu}=(b_1,...,b_d)$,
\item[(b)] a homogeneous tree $W$ with $b_W=b$,
\item[(c)] a node $w_0\in W$,
\item[(d)] a subset $L$ of $\nn$ with $|L|=h(\bfcu)$ and
\item[(e)] for every $j\in\{1,...,q\}$ a level selection $D_j:\otimes\bfcu\to\mathcal{P}(W)$ with $L(D_j)=L$
and such that $w_0\in D_j\big(\bfcu(0)\big)$.
\end{enumerate}
Finally, let $N\in\nn$ with $N\meg k$ and suppose that
\begin{equation} \label{e71}
h(\bfcu) \meg \mil\big(\underbrace{b_1,...,b_1}_{b_1-\mathrm{times}}, ..., \underbrace{b_d,...,b_d}_{b_d-\mathrm{times}}|N,k,b^q\big)+1.
\end{equation}
Then, either
\begin{enumerate}
\item[(i)] there exist $j_0\in\{1,...,q\}$, a vector strong subtree $\bfcq$ of $\bfcu$ and a strong subtree $P$ of $W$ with
$\bfcq(0)=\bfcu(0)$, $P(0)=w_0$, $h(\bfcq)=h(P)=k+1$ and such that
\begin{equation} \label{e72}
P(n)\subseteq \bigcap_{\bfu\in\otimes\bfcq(n)} D_{j_0}(\bfu)
\end{equation}
for every $n\in\{0,...,k\}$, or
\item[(ii)] there exist a vector strong subtree $\bfcu'$ of $\bfcu$ with $\bfcu'(0)=\bfcu(0)$ and of height $N+1$, $p_0\in\{0,...,b-1\}$
and $\mathcal{J}\subseteq\{1,...,q\}$ with $|\mathcal{J}|\meg q/b$ satisfying the following. For $j\in\mathcal{J}$ and every vector strong
subtree $\bfcr$ of $\bfcu'$ of height $k+1$ and with $\bfcr(0)=\bfcu'(0)$ the set
\begin{equation} \label{e73}
\bigcup_{n=1}^{k} \Big(\bigcap_{\bfr\in\otimes\bfcr(n)} D_j(\bfr) \Big)
\end{equation}
does not contain a strong subtree of $\suc_W(w_0^{\con_W}\!p_0)$ of height $k$.
\end{enumerate}
\end{lem}

\subsection{Proof of Lemma \ref{l71}}

Assume that part (i) is not satisfied. This has, in particular, the following consequence.
\medskip

\noindent \textbf{(H)}: \textit{for every $\bfcr\in\strong_{k+1}^0(\bfcu)$ and every $j\in\{1,...,q\}$ there exists $p\in\{0,...,b-1\}$
(depending, possibly, on the choice of $\bfcr$ and $j$) such that the set
\begin{equation} \label{e74}
\bigcup_{n=1}^{k} \Big(\bigcap_{\bfr\in\otimes\bfcr(n)} D_j(\bfr) \Big)
\end{equation}
does not contain a strong subtree of $\suc_W(w_0^{\con_W}\!p)$ of height $k$.}
\medskip

\noindent This assumption permits us to define a coloring $\mathcal{C}:\strong_{k+1}^0(\bfcu)\to \{0,...,b-1\}^q$ by
\begin{eqnarray} \label{e75}
\mathcal{C}(\bfcr)=(p_j)_{j=1}^q & \Leftrightarrow & p_j=\min\{ p: \text{\textbf{(H)} is satisfied for } \bfcr, j \text{ and } p\} \\
& & \text{for every } j\in\{1,...,q\}. \nonumber
\end{eqnarray}
By Corollary \ref{c24} and (\ref{e71}), there exist $\bfcu'\in\strong_{N+1}^0(\bfcu)$ and a finite sequence $(p_j)_{j=1}^q$
in $\{0,...,b-1\}$ such that $\mathcal{C}(\bfcr)=(p_j)_{j=1}^q$ for every $\bfcr\in\strong_{k+1}^0(\bfcu')$. By the classical
pigeonhole principle, there exist $p_0\in\{0,...,b-1\}$ and a subset $\mathcal{J}$ of $\{1,...,q\}$ of cardinality at least
$q/b$ such that $p_j=p_0$ for every $j\in\mathcal{J}$. Thus, with these choices, the second alternative holds true. The proof
of Lemma \ref{l71} is completed.

\subsection{Consequences}

We will need the following consequence of Lemma \ref{l71}.
\begin{cor} \label{c72}
Let $d\in\nn$ with $d\meg 1$ and $b_1,...,b_d\in\nn$ with $b_i\meg 2$ for every $i\in\{1,...,d\}$. Also let $b,k,m\in\nn$
with $b\meg 2$ and $k\meg 1$.

Assume that we are given
\begin{enumerate}
\item[(a)] a finite vector homogeneous tree $\bfcz$ with $b_{\bfcz}=(b_1,...,b_d)$ and $h(\bfcz)\meg m+1$,
\item[(b)] a homogeneous tree $W$ with $b_W=b$,
\item[(c)] a node $w\in W$,
\item[(d)] a nonempty subset $\Delta$ of $\otimes\bfcz(m)$ and
\item[(e)] a level selection $D:\otimes\bfcz\to\mathcal{P}(W)$ such that $w\in D(\bfz)$ for every $\bfz\in\Delta$.
\end{enumerate}
Finally, let $N\in\nn$ with $N\meg k$ and suppose that
\begin{equation} \label{e76}
h(\bfcz) \meg m+ \mil\Big(\underbrace{b_1,...,b_1}_{b_1-\mathrm{times}}, ..., \underbrace{b_d,...,b_d}_{b_d-\mathrm{times}}|N,
k,b^{(\prod_{i=1}^d b_i)^m}\Big)+1.
\end{equation}
Then, either
\begin{enumerate}
\item[(i)] there exist a vector strong subtree $\bfcs$ of $\bfcz$ and a strong subtree $R$ of $W$ with $\bfcs(0)\in\Delta$,
$R(0)=w$, $h(\bfcs)=h(R)=k+1$ and such that
\begin{equation} \label{e77}
R(n)\subseteq \bigcap_{\bfs\in\otimes\bfcs(n)} D(\bfs)
\end{equation}
for every $n\in\{0,...,k\}$, or
\item[(ii)] there exist a vector strong subtree $\bfcz'$ of $\bfcz$ with $\bfcz'\upharpoonright m=\bfcz\upharpoonright m$ and of height
$(m+1)+N$, $p_0\in\{0,...,b-1\}$ and $\Gamma\subseteq\Delta$ with $|\Gamma|\meg (1/b)|\Delta|$ satisfying the following. If $\bfcr'$
is a vector strong subtree of $\bfcz'$ with $\bfcr'(0)\in\Gamma$ and of height $k+1$, then the set
\begin{equation} \label{e78}
\bigcup_{n=1}^{k} \Big(\bigcap_{\bfr\in\otimes\bfcr'(n)} D(\bfr) \Big)
\end{equation}
does not contain a strong subtree of $\suc_W(w^{\con_W}\!p_0)$ of height $k$.
\end{enumerate}
\end{cor}
\begin{proof}
We will argue as in the proof of Corollary \ref{c67}. Specifically, let $\lambda=h(\bfcz)-m$. We write $\bfcz=(Z_1,...,Z_d)$ and we set $\bfcb=(b_1^{<\lambda},...,b_d^{<\lambda})$. For every $\bfz\in\otimes\bfcz(m)$ let $\suc_{\bfcz}(\bfz)=(Z_1^{\bfz},...,Z_d^{\bfz})$
and notice that $h(\bfcb)=h\big(\suc_{\bfcz}(\bfz)\big)=\lambda$. Thus, as in the proof of Corollary \ref{c67}, for every $i\in\{1,...,d\}$
we may consider the canonical isomorphism $\ci_i^{\bfz}$ between $b_i^{<\lambda}$ and $Z_i^{\bfz}$. The vector canonical isomorphism
between $\otimes\bfcb$ and $\otimes\suc_{\bfcz}(\bfz)$ will be denoted by $\bfci_{\bfz}$.

For every $\bfz\in\Delta$ we define the level selection $D_\bfz:\otimes\bfcb\to\mathcal{P}(W)$ exactly as we did in (\ref{e628}).
By (\ref{e76}) and the estimate
\begin{equation} \label{e79}
|\Delta|\mik |\!\otimes\bfcz(m)| = \big( \prod_{i=1}^d b_i\big)^m
\end{equation}
we get that
\begin{equation} \label{e710}
h(\bfcb)\meg \mil\Big(\underbrace{b_1,...,b_1}_{b_1-\mathrm{times}}, ..., \underbrace{b_d,...,b_d}_{b_d-\mathrm{times}}|N,
k,b^{|\Delta|}\Big)+1.
\end{equation}
Hence, we may apply Lemma \ref{l71} to the family $\{D_\bfz:\bfz\in\Delta\}$. If the first alternative of the lemma holds true,
then it is easily seen that part (i) is satisfied.

Otherwise, we get $\bfcu'\in\strong_{N+1}^0(\bfcb)$, $p_0\in\{0,...,b-1\}$ and $\Gamma\subseteq\Delta$ of cardinality at least
$(1/b)|\Delta|$ such that for every $\bfz\in\Gamma$ and every $\bfcr\in\strong_{k+1}^0(\bfcu')$ the set
\begin{equation} \label{e711}
\bigcup_{n=1}^{k} \Big(\bigcap_{\bfr\in\otimes\bfcr(n)} D_\bfz(\bfr) \Big)
\end{equation}
does not contain a strong subtree of $\suc_W(w^{\con_W}\!p_0)$ of height $k$. We have already pointed out in the proof of Corollary \ref{c67}
that for every $i\in\{1,...,d\}$ the family $\{\ci_i^\bfz: \bfz\in\otimes\bfcz(m)\}$ has the following coherence property: for every $\bfz,\bft\in\otimes\bfcz(m)$ if the finite sequences $\bfz$ and $\bft$ agree on the $i$-th coordinate, then the maps $\ci_i^\bfz$ and
$\ci_i^\bft$ are identical. Therefore, it is possible to select a vector strong subtree $\bfcz'$ of $\bfcz$ of height $m+(N+1)$ with
$\bfcz'\upharpoonright m= \bfcz\upharpoonright m$ and such that for every $\bfz\in\Gamma$ we have $\bfci_\bfz(\bfcu')=\suc_{\bfcz'}(\bfz)$.
It is then easy to check that part (ii) is satisfied for $\bfcz'$, $p_0$ and $\Gamma$. The proof is completed.
\end{proof}


\section{Step 3: small correlation}

This section is devoted to the proof of the following result and its consequences.
\begin{lem} \label{l81}
Let $d\in\nn$ with $d\meg 1$ and $b_1,...,b_d,b_{d+1}\in\nn$ with $b_i\meg 2$ for all $i\in\{1,...,d+1\}$. Also let $k\in\nn$ with $k\meg 1$
and assume that for every $0<\eta\mik 1$ the numbers $\ls(b_1,...,b_{d+1}|k,\eta)$ have been defined.

Let $q\in\nn$ with $q\meg 1$ and $0<\eta_0\mik 1$. Assume that we are given
\begin{enumerate}
\item[(a)] a finite vector homogeneous tree $\bfcu$ with $b_{\bfcu}=(b_1,...,b_d)$,
\item[(b)] a homogeneous tree $W$ with $b_W=b_{d+1}$,
\item[(c)] a node $w_0\in W$,
\item[(d)] a subset $L$ of $\nn$ with $|L|=h(\bfcu)$ and $\ell_W(w_0)\mik \min L$, and
\item[(e)] for every $j\in\{1,...,q\}$ a level selection $D_j:\otimes\bfcu\to\mathcal{P}(W)$ with $L(D_j)=L$.
\end{enumerate}
Finally, let $N\in\nn$ with $N\meg 1$ and suppose that
\begin{equation} \label{e81}
h(\bfcu) \meg N + \mil\big(b_1,...,b_d|\ls(b_1,...,b_{d+1}|k,\eta_0), 1,q\big) -1.
\end{equation}
Then, either
\begin{enumerate}
\item[(i)] there exist $j_0\in\{1,...,q\}$, a vector strong subtree $\bfcq$ of $\bfcu$ and a strong subtree $P$ of $\suc_W(w_0)$ with
$h(\bfcq)=h(P)=k$ and such that
\begin{equation} \label{e82}
P(n)\subseteq \bigcap_{\bfu\in\otimes\bfcq(n)} D_{j_0}(\bfu)
\end{equation}
for every $n\in\{0,...,k-1\}$, or
\item[(ii)] there exists a vector strong subtree $\bfcu'$ of $\bfcu$ with $h(\bfcu')=N$ and such that
$\dens\big(D_j\big(\bfcu'(0)\big) \ | \ w_0\big)<\eta_0$ for every $j\in\{1,...,q\}$.
\end{enumerate}
\end{lem}
Lemma \ref{l81} corresponds to the third step of the proof of Theorem \ref{t13}. It will be used, however,
in a more convenient form which is stated and proved in \S 7.2.

\subsection{Proof of Lemma \ref{l81}}

We set
\begin{equation} \label{e83}
N_0=\ls(b_1,...,b_{d+1}|k,\eta_0)
\end{equation}
and
\begin{equation} \label{e84}
M_0=\mil(b_1,...,b_d|N_0,1,q).
\end{equation}
Also let
\begin{equation} \label{e85}
\bfcv=\bfcu\upharpoonright (M_0-1).
\end{equation}
We consider the following cases.
\medskip

\noindent \textsc{Case 1:} \textit{for every $\bfv\in\otimes\bfcv$ there exists $j\in\{1,...,q\}$ with $\dens(D_j(\bfv) \ | \ w_0)\meg\eta_0$}.
In this case we will show that the first alternative of the lemma holds true. Specifically, we define a coloring
$\mathcal{C}:\otimes\bfcv\to \{1,...,q\}$ by the rule
\begin{equation} \label{e86}
\mathcal{C}(\bfv)=\min\big\{ j\in\{1,...,q\} : \dens(D_j(\bfv) \ | \ w_0)\meg\eta_0\big\}.
\end{equation}
Since $h(\bfcv)=M_0$, by the choice of $M_0$ in (\ref{e84}), there exist $j_0\in\{1,...,q\}$ and a vector strong subtree $\bfcv'$ of
$\bfcv$ with $h(\bfcv')=N_0$ such that $\mathcal{C}(\bfv)=j_0$ for every $\bfv\in\otimes\bfcv'$. We define
$D_{j_0}^{w_0}:\otimes\bfcv'\to\mathcal{P}\big(\suc_W(w_0)\big)$ by $D_{j_0}^{w_0}(\bfv)=D_{j_0}(\bfv)\cap\suc_W(w_0)$.
It follows that $D_{j_0}^{w_0}$ is a level selection with density at least $\eta_0$ and of height $N_0$.
Therefore, by the choice of $N_0$ in (\ref{e83}), we conclude that part (i) of the lemma is satisfied.
\medskip

\noindent \textsc{Case 2:} \textit{there is $\bfv_0\in\otimes\bfcv$ such that $\dens(D_j(\bfv_0) \ | \ w_0)<\eta_0$ for every $j\in\{1,...,q\}$}.
By (\ref{e81}) and (\ref{e85}), we have $h\big(\suc_{\bfcu}(\bfv_0)\big)\meg N$. Thus, part (ii) is satisfied for
``$\bfcu'=\suc_{\bfcu}(\bfv_0)\upharpoonright (N-1)$". The proof of Lemma \ref{l81} is completed.

\subsection{Consequences}

We have the following.
\begin{cor} \label{c82}
Let $d\in\nn$ with $d\meg 1$ and $b_1,...,b_d,b_{d+1}\in\nn$ with $b_i\meg 2$ for all $i\in\{1,...,d+1\}$. Also let $k\in\nn$ with $k\meg 1$
and assume that for every $0<\eta\mik 1$ the numbers $\ls(b_1,...,b_{d+1}|k,\eta)$ have been defined.

Let $0<\eta_0\mik 1$ and $m\in\nn$ and define $q=q(b_1,...,b_d,m)$ as in (\ref{e214}). Assume that we are given
\begin{enumerate}
\item[(a)] a finite vector homogeneous tree $\bfcz$ with $b_{\bfcz}=(b_1,...,b_d)$ and $h(\bfcz)\meg m+2$,
\item[(b)] a homogeneous tree $W$ with $b_W=b_{d+1}$,
\item[(c)] a node $\tilde{w}\in W$,
\item[(d)] a subset $L=\{l_0<...< l_{h(\bfcz)-1}\}$ of $\nn$ with $\ell_W(\tilde{w})\mik l_{m+1}$,
\item[(e)] for every $n\in\{0,...,m\}$ a nonempty subset $\Gamma_n$ of $\otimes\bfcz(n)$ and
\item[(f)] a level selection $D:\otimes\bfcz\to\mathcal{P}(W)$ with $L(D)=L$.
\end{enumerate}
Finally, let $N\in\nn$ with $N\meg 1$ and suppose that
\begin{equation} \label{e87}
h(\bfcz) \meg (m+1)+ N + \mil\big(b_1,...,b_d|\ls(b_1,...,b_{d+1}|k,\eta_0), 1,q\big) -1.
\end{equation}
Then, either
\begin{enumerate}
\item[(i)] there exist a vector strong subtree $\bfcs$ of $\bfcz$ with $\bfcs(0)\in\Gamma_0\cup ... \cup \Gamma_m$ and $h(\bfcs)=k+1$
and a strong subtree $R$ of $\suc_W(\tilde{w})$ with $h(R)=k$ such that
\begin{equation} \label{e88}
R(n)\subseteq \bigcap_{\bfs\in\otimes\bfcs(n+1)} D(\bfs)
\end{equation}
for every $n\in\{0,...,k-1\}$, or
\item[(ii)] there exists a vector strong subtree $\bfcz'$ of $\bfcz$ of height $(m+1)+N$ and with
$\bfcz'\upharpoonright m=\bfcz\upharpoonright m$ and satisfying the following. For every $\bfcf'\in\strong_2(\bfcz')$ with
$\bfcf'(0)\in \Gamma_0\cup ... \cup \Gamma_m$ and $\otimes\bfcf'(1)\subseteq \otimes\bfcz'(m+1)$ we have
\begin{equation} \label{e89}
\dens\Big( \bigcap_{\bfz\in\otimes\bfcf'(1)} D(\bfz) \ \big| \ \tilde{w}\Big) <\eta_0.
\end{equation}
\end{enumerate}
\end{cor}
\begin{proof}
We will reduce the proof to Lemma \ref{l81} using the notion of vector canonical isomorphism exactly as we did in the proofs of Corollary
\ref{c67} and Corollary \ref{c72}. The reduction in this case is slightly more involved but the overall strategy is identical.

Specifically, let $\lambda=h(\bfcz)-(m+1)$. We write $\bfcz=(Z_1,...,Z_d)$ and we set $\bfcb=(b_1^{<\lambda},...,b_d^{<\lambda})$.
Also, for every $\bfz\in\otimes\bfcz(m+1)$ we set $\suc_{\bfcz}(\bfz)=(Z_1^{\bfz},...,Z_d^{\bfz})$. As in the proof of Corollary \ref{c67},
for every $i\in\{1,...,d\}$ by $\ci_i^{\bfz}$ we shall denote the canonical isomorphism between $b_i^{<\lambda}$ and $Z_i^{\bfz}$. The vector
canonical isomorphism between $\otimes\bfcb$ and $\otimes\suc_{\bfcz}(\bfz)$ will be denoted by $\bfci_{\bfz}$. We define
\begin{equation} \label{e810}
\mathcal{F}=\big\{\bfcf\in\strong_2(\bfcz): \bfcf(0)\in \Gamma_0\cup ... \cup \Gamma_m \text{ and } \otimes\bfcf(1)\subseteq\otimes\bfcz(m+1)\big\}.
\end{equation}
By Fact \ref{f21} and the choice of $q$, we have the estimate
\begin{equation} \label{e811}
|\mathcal{F}|\mik q.
\end{equation}
For every $\bfcf\in\mathcal{F}$ we define $D_{\bfcf}:\otimes\bfcb\to\mathcal{P}(W)$ by the rule
\begin{equation} \label{e812}
D_\bfcf(\bfu)= \bigcap_{\bfz\in\otimes\bfcf(1)} D\big(\bfci_{\bfz}(\bfu)\big).
\end{equation}
Notice that $D_{\bfcf}$ is a level selection with $L(D_{\bfcf})=\{l_{m+1}<...<l_{h(\bfcz)-1}\}$. Moreover,
\begin{eqnarray} \label{e813}
h(\bfcb) & \stackrel{(\ref{e87})}{\meg} & N + \mil\big(b_1,...,b_d|\ls(b_1,...,b_{d+1}|k,\eta_0), 1,q\big) -1 \\
& \stackrel{(\ref{e811})}{\meg} & N + \mil\big(b_1,...,b_d|\ls(b_1,...,b_{d+1}|k,\eta_0), 1,|\mathcal{F}|\big) -1. \nonumber
\end{eqnarray}
Therefore, by Lemma \ref{l81} applied to the family $\{D_{\bfcf}:\bfcf\in\mathcal{F}\}$, we get that one of the following cases is satisfied.
\medskip

\noindent \textsc{Case 1:} \textit{there exist $\bfcg\in\mathcal{F}$, a vector strong subtree $\bfcq$ of $\bfcb$ and a strong subtree
$R$ of $\suc_W(\tilde{w})$ with $h(\bfcq)=h(R)=k$ and such that
\begin{equation} \label{e814}
R(n) \subseteq \bigcap_{\bfu\in\otimes\bfcq(n)} D_{\bfcg}(\bfu)
\end{equation}
for every $n\in\{0,...,k-1\}$}. In this case, we will show that the first alternative of the corollary holds true. Specifically, write
$\bfcg=(G_1,...,G_d)$ and $\bfcq=(Q_1,...,Q_d)$. Since $\bfcg\in\mathcal{F}$ there exists $m_0\in \{0,...,m\}$ such that
$\bfcg(0)\in\Gamma_{m_0}$. For every $i\in\{1,...,d\}$ we set
\begin{equation} \label{e815}
S_i= G_i(0) \cup \bigcup_{\bfz\in\otimes\bfcg(1)} \ci_i^{\bfz}(Q_i)
\end{equation}
and we define
\begin{equation} \label{e816}
\bfcs=(S_1,...,S_d).
\end{equation}
It is easy to see that $\bfcs$ is a vector strong subtree of $\bfcz$ with $h(\bfcs)=k+1$ and such that $\bfcs(0)=\bfcg(0)\in\Gamma_{m_0}$.
Moreover, we have the following.
\begin{fact} \label{f83}
We have $\otimes\bfcs(n+1)=\{\bfci_{\bfz}(\bfu): \bfu\in\otimes\bfcq(n) \text{ and } \bfz\in\otimes\bfcg(1)\}$ for every $n\in\{0,...,k-1\}$.
\end{fact}
Fact \ref{f83} is a rather straightforward consequence of the relevant definitions. Now let $n\in\{0,...,k-1\}$ be arbitrary.
By (\ref{e814}), (\ref{e812}) and Fact \ref{f83}, we have
\begin{equation} \label{e817}
R(n) \subseteq \bigcap_{\bfu\in\otimes\bfcq(n)} \bigcap_{\bfz\in\otimes\bfcg(1)} D\big(\bfci_{\bfz}(\bfu)\big) =
\bigcap_{\bfs\in\otimes\bfcs(n+1)} D(\bfs).
\end{equation}
Thus, in this case part (i) of the corollary is satisfied.
\medskip

\noindent \textsc{Case 2:} \textit{there exists a vector strong subtree $\bfcu'$ of $\bfcb$ with $h(\bfcu')=N$ and such that
\begin{equation} \label{e818}
\dens\big( D_{\bfcf}\big(\bfcu'(0)\big) \ | \ \tilde{w}\big) <\eta_0
\end{equation}
for every $\bfcf\in\mathcal{F}$}. As the reader might have already guessed, we will show that the second alternative of the corollary
holds true. To this end write $\bfcu'=(U_1',..., U_d')$.
For every $i\in\{1,...,d\}$ we set
\begin{equation} \label{e819}
Z_i' = (Z_i\upharpoonright m) \cup \bigcup_{\bfz\in\otimes\bfcz(m+1)} \ci_i^\bfz(U_i')
\end{equation}
and we define
\begin{equation} \label{e820}
\bfcz'=(Z_1',...,Z_d').
\end{equation}
It is easy to check that $\bfcz'$ is a vector strong subtree of $\bfcz$ of height $(m+1)+N$ and such that $\bfcz'\upharpoonright m=
\bfcz\upharpoonright m$. Also notice that there exists a natural bijection $\phi$ between $\otimes\bfcz(m+1)$ and $\otimes\bfcz'(m+1)$.
It is defined by the rule
\begin{equation} \label{e821}
\phi(\bfz) = \bfci_{\bfz}\big(\bfcu'(0)\big).
\end{equation}
Observe that the map $\phi$ has the following coherence property: if $\bfcf'\in\strong_2(\bfcz')$ with $\otimes\bfcf'(1)\subseteq
\otimes\bfcz'(m+1)$, then there exists $\bfcf\in\strong_2(\bfcz)$ with $\bfcf'(0)=\bfcf(0)$, $\otimes\bfcf(1)\subseteq \otimes\bfcz(m+1)$
and such that $\phi\big(\!\otimes\bfcf(1)\big)=\otimes\bfcf'(1)$. These remarks yield the following fact.
\begin{fact} \label{f84}
Let $\bfcf'\in\strong_2(\bfcz')$ with $\bfcf'(0)\in\Gamma_0\cup ... \cup \Gamma_m$ and $\otimes\bfcf'(1)\subseteq
\otimes\bfcz'(m+1)$. Then there exists $\bfcf\in\mathcal{F}$ such that $\otimes\bfcf'(1)= \big\{  \bfci_{\bfz}\big(\bfcu'(0)\big):
\bfz\in\otimes\bfcf(1)\big\}$.
\end{fact}
The vector strong subtree $\bfcz'$ of $\bfcz$ satisfies all requirements of part (ii). Indeed, we have already pointed out that
$h(\bfcz')=(m+1)+N$ and $\bfcz'\upharpoonright m= \bfcz\upharpoonright m$. Now let $\bfcf'\in\strong_2(\bfcz')$ with
$\bfcf'(0)\in\Gamma_0\cup ... \cup \Gamma_m$ and $\otimes\bfcf'(1)\subseteq \otimes\bfcz'(m+1)$ be arbitrary.
By Fact \ref{f83}, there exists  $\bfcf\in\mathcal{F}$ such that $\otimes\bfcf'(1)= \big\{  \bfci_{\bfz}\big(\bfcu'(0)\big):
\bfz\in\otimes\bfcf(1)\big\}$. It follows that
\begin{eqnarray} \label{e822}
\dens\Big( \bigcap_{\bfz\in\otimes\bfcf'(1)} D(\bfz) \ \big| \ \tilde{w}\Big) & = &
\dens\Big( \bigcap_{\bfz\in\otimes\bfcf(1)}  D\Big(\bfci_{\bfz}\big(\bfcu'(0)\big)\Big) \ \big| \ \tilde{w}\Big) \\
& \stackrel{(\ref{e812})}{=} & \dens\big( D_{\bfcf}\big(\bfcu'(0)\big) \ | \ \tilde{w}\big)
\stackrel{(\ref{e818})}{<} \eta_0. \nonumber
\end{eqnarray}
The proof of Corollary \ref{c82} is completed.
\end{proof}


\section{Performing the algorithm}

Our goal in this section is to perform the algorithm outlined in \S 4 using the analysis of the three basic steps given in \S 5, \S 6
and \S 7. We recall that this algorithm is part of the inductive scheme described in (\ref{e41}). In particular, we make the following
assumptions which will be repeatedly used throughout this section.
\medskip

\noindent \textbf{Assumptions.} \textit{We fix $d\in\nn$ with $d\meg 1$, $b_1,...,b_d,b_{d+1}\in\nn$ with $b_i\meg 2$ for all
$i\in\{1,...,d+1\}$, $k\in\nn$ with $k\meg 1$ and $0<\ee\mik 1$. We will assume that for every integer $\ell\meg 1$ and every real
$0<\eta\mik 1$ the numbers $\udhl(b_1,...,b_d|\ell,\eta)$ and $\ls(b_1,...,b_{d+1}|k,\eta)$ have been defined. }

\subsection{Initializing various numerical parameters}

First we set
\begin{equation} \label{e91}
K_0=\udhl\big(b_1,...,b_d|2,\ee/(4b_{d+1})\big).
\end{equation}
The number $K_0$ is the number of iterations of the algorithm. Next we set
\begin{equation} \label{e92}
r=\Big( \frac{\ee}{16 (\prod_{i=1}^d b_i)^{K_0}b_{d+1}}\Big)^{2^{3K_0-1}}.
\end{equation}
The quantity $r$ will be used to control the density increment. Finally let
\begin{equation} \label{e93}
Q_0=\frac{\big( \prod_{i=1}^d b_i^{b_i} \big)^{K_0} -\big( \prod_{i=1}^d b_i \big)^{K_0}}{\prod_{i=1}^d b_i^{b_i} -\prod_{i=1}^d b_i}
\end{equation}
and set
\begin{equation} \label{e94}
\theta_0=\frac{\ee}{8Q_0}.
\end{equation}
The quantity $\theta_0$ will be used to quantify what ``negligible" means in the third step of each iteration.

\subsection{Functions that control the height}

Each time we perform a basic step of the algorithm we refine the finite vector homogeneous tree that we have as input to achieve further
properties. The height of the resulting vector strong subtree will be controlled by three functions $f_1, f_2$ and $f_3$ corresponding to the
first, second and third step respectively.

Specifically, we define $f_1:\nn\to\nn$ by $f_1(0)=0$ and
\begin{equation} \label{e95}
f_1(n) =  \lceil 1/r^3\rceil \udhl(b_1,...,b_d|n,r^3)
\end{equation}
for every integer $n\meg 1$. Next let $f_2:\nn\to\nn$ be defined by $f_2(n)=0$ if $n<k$ and
\begin{equation} \label{e96}
f_2(n) = \mil\Big(\underbrace{b_1,...,b_1}_{b_1-\mathrm{times}}, ..., \underbrace{b_d,...,b_d}_{b_d-\mathrm{times}}|n,
k,b_{d+1}^{(\prod_{i=1}^d b_i)^{K_0-1}}\Big)
\end{equation}
for every integer $n\meg k$. Also we define $f_3:\nn\to\nn$ by
\begin{equation} \label{e97}
f_3(n)=\mil\big(b_1,...,b_d|\ls(b_1,...,b_{d+1}|k,\theta_0),1,Q_0\big)+n-1.
\end{equation}
Finally, let $g:\nn\to\nn$ be defined by the rule
\begin{equation} \label{e98}
g(n)= (f_1\circ f_2\circ f_3)(n)+1.
\end{equation}
It is, of course, clear that the function $g$ will be used to control the height of the resulting vector strong subtree
after all steps have been performed.

\subsection{Control of loss of density}

Recall that, by Lemma \ref{l62}, if we have ``lack of density increment" then it is always possible to have strong denseness loosing just
a small amount of density. This basic fact will be repeatedly used while performing the algorithm and is appropriately quantified as follows.

We define $(\delta_n)_{n=0}^{3K_0-1}$ and $(\ee_n)_{n=0}^{K_0}$ recursively by the rule
\begin{equation} \label{e99}
\left\{ \begin{array} {l} \delta_0=r, \\ \delta_{n+1} =(\delta_n+\delta_n^2)^{1/2} \end{array}  \right. \text{ and } \ \
\left\{ \begin{array} {l} \ee_0=\ee, \\ \ee_{n+1} =\ee_n-(\delta_{3n}+\delta_{3n+1}+\delta_{3n+2}). \end{array}  \right.
\end{equation}
The sequence $(\delta_n)_{n=0}^{3K_0-1}$ will be used to control the loss of density at each step of the iteration while the sequence
$(\ee_n)_{n=0}^{K_0}$ stands for the density left at our disposal. We will need the following properties satisfied by these sequences.
\begin{enumerate}
\item[($\mathcal{P}$1)] For every $n\in\{0,...,3K_0-1\}$ we have $\delta_n\mik 2 r^{2^{-n}}$.
\item[($\mathcal{P}$2)] For every $n\in\{0,...,3K_0-2\}$ we have $\sum_{i=0}^{n} \delta_i =\delta_{n+1}^2-r^2$.
\item[($\mathcal{P}$3)] We have $\sum_{n=0}^{3K_0-1} \delta_n \mik \ee/2$.
\item[($\mathcal{P}$4)] For every $n\in\{0,...,K_0-1\}$ we have $\ee_{n+1}= \ee - \sum_{i=0}^{3n+2} \delta_i$.
\item[($\mathcal{P}$5)] For every $n\in\{0,...,K_0\}$ we have $\ee/2 \mik \ee_n\mik \ee$.
\item[($\mathcal{P}$6)] For every $n\in\{0,...,K_0-1\}$ we have
\begin{equation} \label{e910}
\delta_{3n}\mik \delta_{3K_0-3} \mik  \Big( \frac{\ee/2}{4(\prod_{i=1}^d b_i)^{K_0}b_{d+1}}\Big)^4 \mik
\Big( \frac{\ee/2}{4(\prod_{i=1}^d b_i)^{n+1}b_{d+1}}\Big)^4.
\end{equation}
\end{enumerate}
The verification of these properties is fairly elementary and is left to the reader. We simply notice that properties
($\mathcal{P}$3) and ($\mathcal{P}$6) follow by the choice of $r$ in (\ref{e92}).

\subsection{The main dichotomy}

We are now ready to state the main result in this section which is the last step towards the proof of Theorem \ref{t13}.
\begin{lem} \label{l91}
Let $\bfct$ be a finite vector homogeneous tree with $b_{\bfct}=(b_1,...,b_d)$, $W$ a homogeneous tree with $b_W=b_{d+1}$
and $D:\otimes\bfct\to\mathcal{P}(W)$ a level selection with $\delta(D)\meg \ee$. Also let $M\in\nn$ with $M\meg k$
and assume that
\begin{equation} \label{e911}
h(\bfct)\meg g^{(K_0)}(M).
\end{equation}
Then, either
\begin{enumerate}
\item[(i)] there exist a vector strong subtree $\bfct'$ of $\bfct$ with $h(\bfct')=M$ and $w'\in W$ such that $D$ is
$(w',\bfct',\ee+r^2/2)$-dense, or
\item[(ii)] there exist a vector strong subtree $\bfcs$ of $\bfct$ and a strong subtree $R$ of $W$ with $h(\bfcs)=h(R)=k+1$
and such that for every $n\in\{0,...,k\}$ we have
\begin{equation} \label{e912}
R(n)\subseteq \bigcap_{\bfs\in\otimes\bfcs(n)} D(\bfs).
\end{equation}
\end{enumerate}
\end{lem}

\subsection{Proof of Lemma \ref{l91}}

Assuming that neither (i) nor (ii) are satisfied we will derive a contradiction. In particular, recursively we will construct
\begin{enumerate}
\item[(a)] a finite sequence $(\bfcz_n)_{n=0}^{K_0-1}$ of vector strong subtrees of $\bfct$,
\item[(b)] a finite sequence $(A_n)_{n=0}^{K_0-1}$ of subsets of $W$,
\item[(c)] two finite sequences $(w_n)_{n=0}^{K_0-1}$ and $(\tilde{w}_n)_{n=0}^{K_0-1}$ of nodes of $W$,
\item[(d)] two finite sequences $(\Delta_n)_{n=0}^{K_0-1}$ and $(\Gamma_n)_{n=0}^{K_0-1}$ of subsets of $\otimes\bfct$ and
\item[(e)] a strictly increasing finite sequence $(\ell_n)_{n=0}^{K_0-1}$ in $\nn$
\end{enumerate}
such that, setting
\begin{equation} \label{e913}
\mathcal{F}_n=\big\{ \bfcf\in\strong_2(\bfcz_n): \bfcf(0)\in\Gamma_0\cup ... \cup \Gamma_n \text{ and }
\!\otimes\bfcf(1)\subseteq \otimes\bfcz_n(n+1)\big\},
\end{equation}
for every $n\in\{0,...,K_0-1\}$ the following conditions are satisfied.
\begin{enumerate}
\item[(C1)] We have $\ell_0=\min L(D)$, $A_0=D\big(\bfct(0)\big)$ and $\Delta_0=\Gamma_0=\bfct(0)$.
\item[(C2)] We have $h(\bfcz_n)=n+1+g^{(K_0-n-1)}(M)$ and $\bfcz_n(0)=\bfct(0)$. Moreover, if $n\meg 1$, then
$\bfcz_n\upharpoonright n=\bfcz_{n-1}\upharpoonright n$.
\item[(C3)] For every $\bfz\in\otimes\bfcz_n(n)$ we have $D(\bfz)\subseteq W(\ell_n)$.
\item[(C4)] The level selection $D$ is $(\tilde{w}_n,\suc_{\bfcz_n}(\bfz),\ee_{n+1})$-dense for all $\bfz\in\otimes\bfcz_n(n+1)$.
\item[(C5)] If $n\meg 1$, then $A_n\subseteq \suc_W(\tilde{w}_{n-1})\cap W(\ell_n)$ and $\dens(A_n \ | \ \tilde{w}_{n-1})\meg\ee/8$.
\item[(C6)] If $n\meg 1$, then $\Delta_n$ is a subset of $\otimes\bfcz_{n-1}(n)$ of cardinality $(\ee/4)|\!\otimes\bfcz_{n-1}(n)|$.
\item[(C7)] We have $\Gamma_n\subseteq \Delta_n$ with $|\Gamma_n|\meg (1/b_{d+1})|\Delta_n|$.
\item[(C8)] We have
\begin{equation} \label{e914}
\dens\Big( \bigcup_{\bfcf\in\mathcal{F}_n} \bigcap_{\bfz\in\otimes\bfcf(1)} D(\bfz) \ \big| \ \tilde{w}_n\Big)\mik \ee/8.
\end{equation}
\item[(C9)] If $n\meg 1$, then
\begin{equation} \label{e915}
A_n\cap \Big( \bigcup_{\bfcf\in\mathcal{F}_{n-1}} \bigcap_{\bfz\in\otimes\bfcf(1)} D(\bfz)\Big)=\varnothing.
\end{equation}
\item[(C10)] We have
\begin{equation} \label{e916}
w_n\in A_n\cap \bigcap_{\bfz\in \Delta_n} D(\bfz).
\end{equation}
\item[(C11)] We have $\tilde{w}_n\in\immsuc_W(w_n)$.
\item[(C12)] For every $\bfcr\in\strong_{k+1}(\bfcz_n)$ with $\bfcr(0)\in \Gamma_0\cup ...\cup \Gamma_n$ the set
\begin{equation} \label{e917}
\bigcup_{j=1}^k \bigcap_{\bfz\in\otimes\bfcr(j)} D(\bfz)
\end{equation}
does not contain a strong subtree of $\suc_W(\tilde{w}_n)$ of height $k$.
\end{enumerate}
As the first step is identical to the general one, let $n\in\{0,...,K_0-2\}$ and assume that the construction has been carried
out up to $n$ so that the above conditions are satisfied.

\subsubsection*{Step 1: selection of $\ell_{n+1}, A_{n+1}, w_{n+1}$ and $\Delta_{n+1}$}

Let $\ell_{n+1}$ be the unique element of the level set $L(D)$ of $D$ such that
\begin{equation} \label{e918}
D(\bfz)\subseteq W(\ell_{n+1})
\end{equation}
for every $\bfz\in\otimes\bfcz_n(n+1)$. For every $w\in W(\ell_{n+1})$ we set
\begin{equation} \label{e919}
\Delta_w=\big\{\bfz\in\otimes\bfcz_n(n+1): w\in D(\bfz)\big\}
\end{equation}
and we define
\begin{equation} \label{e920}
B_{n+1}=\big\{w\in\suc_W(\tilde{w}_n)\cap W(\ell_{n+1}): |\Delta_w|\meg (\ee_{n+1}/2)|\!\otimes\bfcz_n(n+1)|\big\}.
\end{equation}
By condition (C4), the level selection $D$ is $(\tilde{w}_n,\suc_{\bfcz_n}(\bfz),\ee_{n+1})$-dense for every $\bfz\in\otimes\bfcz_n(n+1)$.
Therefore,
\begin{equation} \label{e921}
\dens(B_{n+1} \ | \ \tilde{w}_n)\meg \ee_{n+1}/2 \stackrel{(\mathcal{P}5)}{\meg} \ee/4.
\end{equation}
Next we set
\begin{equation} \label{e922}
A_{n+1}= B_{n+1} \setminus \Big( \bigcup_{\bfcf\in\mathcal{F}_n} \bigcap_{\bfz\in\otimes\bfcf(1)} D(\bfz)\Big).
\end{equation}
Using estimates (\ref{e914}) and (\ref{e921}), we see that with these choices conditions (C5) and (C9) are satisfied.

We proceed to select the node $w_{n+1}$ and the set $\Delta_{n+1}$. This will be done with an appropriate application of Corollary \ref{c67}.
Specifically, we set
\begin{equation} \label{e923}
M_2=(f_2\circ f_3)\big( g^{(K_0-n-2)}(M)\big).
\end{equation}
Notice that $M_2\meg M\meg k\meg 1$. Moreover,
\begin{eqnarray} \label{e924}
h(\bfcz_n) & \stackrel{(\mathrm{C2})}{=} & (n+1) + g^{(K_0-n-1)}(M) \\
& \stackrel{(\ref{e98})}{=} & (n+2) + (f_1\circ f_2\circ f_3) \big( g^{(K_0-n-2)}(M)\big) \nonumber \\
& \stackrel{(\ref{e923})}{=} & (n+2) + f_1(M_2) \nonumber \\
& \stackrel{(\ref{e95})}{\meg} & (n+2) + \frac{1}{r^3}\udhl(b_1,...,b_d|M_2,r^3). \nonumber
\end{eqnarray}
Also,
\begin{eqnarray} \label{e925}
\gamma_0(\ee_{n+1},\ee,r) & \stackrel{(\ref{e62})}{=} & (\ee+r^2-\ee_{n+1})^{1/2} \\
& \stackrel{(\mathcal{P}4)}{=} & \Big( \sum_{i=0}^{3n+2} \delta_i +r^2 \Big)^{1/2} \stackrel{(\mathcal{P}2)}{=}
\delta_{3(n+1)}. \nonumber
\end{eqnarray}
By (\ref{e63}), (\ref{e64}), (\ref{e99}) and the above identity, we see that
\begin{equation} \label{e926}
\gamma_1(\ee_{n+1},\ee,r)=\delta_{3(n+1)+1} \ \text{ and } \ \gamma_2(\ee_{n+1},\ee,r)=\delta_{3(n+1)+2}.
\end{equation}
Finally, by properties ($\mathcal{P}$5) and ($\mathcal{P}$6), we get
\begin{equation} \label{e927}
\delta_{3(n+1)}\mik \Big( \frac{\ee_{n+1}}{4(\prod_{i=1}^d b_i)^{n+2}b_{d+1}}\Big)^4.
\end{equation}
It follows by condition (C4) and the above discussion that we may apply Corollary \ref{c67} for ``$\alpha=\ee_{n+1}$",
``$\beta=\ee$", ``$\varrho=r$", ``$\gamma_i=\delta_{3(n+1)+i}$" for $i\in\{0,1,2\}$, ``$b=b_{d+1}$", ``$m=n+1$", ``$\bfcz=\bfcz_n$",
``$\tilde{w}=\tilde{w}_n$", ``$\ell=\ell_{n+1}$", ``$A=A_{n+1}$", ``$D=D\upharpoonright \bfcz_n$" and ``$N=M_2$" (for the first
step of the recursive selection we set ``$\bfcz=\bfct$", ``$\tilde{w}=W(0)$" and ``$m=1$"; the rest of the parameters are chosen
mutatis mutandis taking into account the choices we made in condition (C1) of the recursive construction). We have already pointed
out that $M_2\meg M$. Since we have assumed that part (i) of the lemma is not satisfied, we see that the second alternative of Corollary
\ref{c67} holds true. Therefore, there exist a vector strong subtree $\bfcv_1$ of $\bfcz_n$ and a node $w\in A_{n+1}$ such that
\begin{enumerate}
\item[(1a)] $h(\bfcv_1)=(n+2)+M_2$,
\item[(1b)] $\bfcv_1\upharpoonright (n+1)=\bfcz_n\upharpoonright (n+1)$ and
\item[(1c)] $D$ is $(w,\suc_{\bfcv_1}(\bfz),\ee_{n+2})$-strongly dense for every $\bfz\in\otimes\bfcv_1(n+2)$.
\end{enumerate}
We set
\begin{equation} \label{e928}
w_{n+1}=w \ \text{ and } \ \Delta_{n+1}=\Delta_w
\end{equation}
and we observe that with these choices condition (C6) and (C10) are satisfied. The first step of the recursive selection
is completed and, so far, conditions (C5), (C6), (C9) and (C10) are satisfied.

\subsubsection*{Step 2: selection of $\Gamma_{n+1}$ and $\tilde{w}_{n+1}$}

In this step we will rely on Corollary \ref{c72}. Precisely, we set
\begin{equation} \label{e929}
M_1= f_3\big( g^{(K_0-n-2)}(M)\big)
\end{equation}
and we observe that $M_1\meg M\meg k$. Moreover,
\begin{eqnarray} \label{e930}
\ \ \ \ \ \ \ h(\bfcv_1) & \stackrel{(\mathrm{1a})}{=} & (n+2) + M_2 \\
& \stackrel{(\ref{e923})}{=} & (n+2) + (f_2\circ f_3)\big( g^{(K_0-n-2)}(M)\big) \nonumber \\
& \stackrel{(\ref{e929})}{=} & (n+2) + f_2(M_1) \nonumber \\
& \stackrel{(\ref{e96})}{\meg} & (n+1) +
\mil\Big(\underbrace{b_1,...,b_1}_{b_1-\mathrm{times}}, ..., \underbrace{b_d,...,b_d}_{b_d-\mathrm{times}}|M_1,
k,b_{d+1}^{(\prod_{i=1}^d b_i)^{n+1}}\Big)+1. \nonumber
\end{eqnarray}
Using this estimate and the fact that condition (C10) has already been verified for $n+1$, we may apply Corollary \ref{c72} for ``$b=b_{d+1}$",
``$m=n+1$", ``$\bfcz=\bfcv_1$", ``$w=w_{n+1}$", ``$\Delta=\Delta_{n+1}$", ``$D=D\upharpoonright \bfcv_1$" and ``$N=M_1$". Recall that, by our
assumptions, part (ii) of the lemma is not satisfied. It follows that the second alternative of Corollary \ref{c72} holds true. Therefore, there
exist a vector strong subtree $\bfcv_2$ of $\bfcv_1$, $p_0\in\{0,...,b_{d+1}-1\}$ and a subset $\Gamma$ of $\Delta_{n+1}$ such that
\begin{enumerate}
\item[(2a)] $h(\bfcv_2)=(n+2)+M_1$,
\item[(2b)] $\bfcv_2\upharpoonright (n+1)=\bfcv_1\upharpoonright (n+1)$,
\item[(2c)] $|\Gamma|\meg (1/b_{d+1}) |\Delta_{n+1}|$ and
\item[(2d)] for every $\bfcr\in\strong_{k+1}(\bfcv_2)$ with $\bfcr(0)\in\Gamma$ the set
\begin{equation} \label{e931}
\bigcup_{j=1}^k \bigcap_{\bfz\in\otimes\bfcr(j)} D(\bfz)
\end{equation}
does not contain a strong subtree of $\suc_W(w_{n+1}^{\con_W} p_0)$ of height $k$.
\end{enumerate}
We set
\begin{equation} \label{e932}
\tilde{w}_{n+1}=w_{n+1}^{\con_W} p_0 \ \text{ and } \ \Gamma_{n+1}=\Gamma.
\end{equation}
and we observe that with these choices conditions (C7) and (C11) are satisfied. The second step of the recursive selection
is completed.

\subsubsection*{Step 3: selection of $\bfcz_{n+1}$}

As the reader might have already guessed, the selection of $\bfcz_{n+1}$ will be achieved with the help of Corollary \ref{c82}.
To apply Corollary \ref{c82}, however, we need to do some preparatory work.

Firstly, we will use our inductive hypotheses to strengthen property (2d) above. Specifically, by (1b) and (2b), we have that
$\bfcv_2$ is a vector strong subtree of $\bfcz_n$ with $\bfcv_2\upharpoonright (n+1)=\bfcz_n\upharpoonright(n+1)$. Moreover, $\tilde{w}_{n+1}\in\immsuc_W(w_{n+1})$ and
\begin{equation} \label{e933}
w_{n+1}\in A_{n+1} \stackrel{(\ref{e922})}{\subseteq} B_{n+1} \stackrel{(\ref{e920})}{\subseteq} \suc_W(\tilde{w}_n).
\end{equation}
Taking into account these remarks and using condition (C12) for $\bfcz_n$, we arrive at the following.
\begin{enumerate}
\item[(2e)] For every $\bfcr\in\strong_{k+1}(\bfcv_2)$ with $\bfcr(0)\in\Gamma_0\cup ...\cup \Gamma_{n+1}$ the set
\begin{equation} \label{e934}
\bigcup_{j=1}^k \bigcap_{\bfz\in\otimes\bfcr(j)} D(\bfz)
\end{equation}
does not contain a strong subtree of $\suc_W(\tilde{w}_{n+1})$ of height $k$.
\end{enumerate}
Next we set
\begin{equation} \label{e935}
q_n=q(b_1,...,b_d,n+1) \stackrel{(\ref{e214})}{=}\frac{\big( \prod_{i=1}^d b_i^{b_i} \big)^{n+2} -
\big( \prod_{i=1}^d b_i \big)^{n+2}}{\prod_{i=1}^d b_i^{b_i} -\prod_{i=1}^d b_i}
\end{equation}
and we notice that
\begin{equation} \label{e936}
q_n \stackrel{(\ref{e93})}{\mik} Q_0.
\end{equation}
Finally, let
\begin{equation} \label{e937}
M_0=g^{(K_0-n-2)}(M)
\end{equation}
and observe that
\begin{eqnarray} \label{e938}
\ \ \ \ \ \ \ h(\bfcv_2) & \stackrel{(\mathrm{2a})}{=} & (n+2) + M_1 \stackrel{(\ref{e929})}{=} (n+2) + f_3\big( g^{(K_0-n-2)}(M)\big) \\
& \stackrel{(\ref{e937})}{=} & (n+2)+ f_3(M_0) \nonumber \\
& \stackrel{(\ref{e97})}{=} & (n+2) + M_0 + \mil\big(b_1,...,b_d|\ls(b_1,...,b_{d+1}|k,\theta_0),1,Q_0\big) -1. \nonumber \\
& \stackrel{(\ref{e936})}{\meg} &  (n+2) + M_0 + \mil\big(b_1,...,b_d|\ls(b_1,...,b_{d+1}|k,\theta_0),1,q_n\big) -1. \nonumber
\end{eqnarray}
Therefore, we may apply Corollary \ref{c82} for ``$\eta_0=\theta_0$", ``$m=n+1$", ``$\bfcz=\bfcv_2$", ``$\tilde{w}=\tilde{w}_{n+1}$",
the family ``$\{\Gamma_0,..., \Gamma_{n+1}\}$", ``$D=D\upharpoonright\bfcv_2$" and ``$N=M_0$". The first alternative of Corollary
\ref{c82} contradicts (2e) isolated above. Thus, there exists a vector strong subtree $\bfcv_3$ of $\bfcv_2$ such that
\begin{enumerate}
\item[(3a)] $h(\bfcv_3)=(n+2)+M_0\stackrel{(\ref{e937})}{=} (n+2) + g^{(K_0-n-2)}(M)$,
\item[(3b)] $\bfcv_3\upharpoonright (n+1)=\bfcv_2\upharpoonright (n+1)$ and
\item[(3c)] for every $\bfcf\in\strong_2(\bfcv_3)$ with $\bfcf(0)\in\Gamma_0\cup ...\cup \Gamma_{n+1}$ and $\otimes\bfcf(1)
\subseteq \otimes\bfcv_3(n+2)$ we have
\begin{equation} \label{e939}
\dens\Big( \bigcap_{\bfz\in\otimes\bfcf(1)} D(\bfz) \ \big| \ \tilde{w}_{n+1}\Big) <\theta_0.
\end{equation}
\end{enumerate}
We set
\begin{equation} \label{e940}
\bfcz_{n+1}=\bfcv_3
\end{equation}
and we claim that with this choice all conditions are satisfied. Recall that conditions (C5), (C6), (C7), (C9), (C10) and (C11) have already
been verified and so we only have to argue for the rest. Condition (C1) is just an initial assumption. Conditions (C2) and (C3) follow by
(\ref{e918}), (1b), (2b), (3a) and (3b). Condition (C4) follows by (1c) and the fact that $\bfcz_{n+1}$ is a vector strong subtree of $\bfcv_1$.
To see that condition (C8) is satisfied notice that, by Fact \ref{f21}, we have
\begin{equation} \label{e941}
|\mathcal{F}_{n+1}| \stackrel{(\ref{e214})}{\mik} q(b_1,...,b_d,n+1) \stackrel{(\ref{e93})}{\mik} Q_0.
\end{equation}
Hence,
\begin{equation} \label{e942}
\dens\Big( \bigcup_{\bfcf\in\mathcal{F}_{n+1}}\bigcap_{\bfz\in\otimes\bfcf(1)} D(\bfz) \ \big| \ \tilde{w}_{n+1}\Big)
\stackrel{(\ref{e939}),(\ref{e941})}{\mik} \theta_0 Q_0  \stackrel{(\ref{e94})}{\mik} \ee/8.
\end{equation}
Finally, condition (C12) follows by (2e) and the fact that $\bfcz_{n+1}$ is a vector strong subtree of $\bfcv_2$.
The recursive selection is completed.

\subsubsection*{Getting the contradiction}

We are now in a position to derive the contradiction. By condition (C2), we have $h(\bfcz_{K_0-1})\meg K_0$. We set
$\bfcb=\bfcz_{K_0-1}\upharpoonright (K_0-1)$. By conditions (C1), (C2), (C6) and (C7), we see that $\Gamma_n$ is a subset of
$\otimes\bfcb(n)$ of cardinality at least $(\ee/4b_{d+1})|\otimes\bfcb(n)|$ for every $n\in\{0,...,K_0-1\}$. Next observe that
$b_{\bfcb}=(b_1,...,b_d)$. Hence, by the choice of $K_0$ in (\ref{e91}), there exist $\bfcg\in\strong_2(\bfcb)$
and $0\mik n_0<n_1<K_0$ such that $\bfcg(0)\in\Gamma_{n_0}$ and $\otimes\bfcg(1)\subseteq \Gamma_{n_1}$.

By condition (C2) and the choice of $\bfcb$, we have that $\bfcb\upharpoonright n_1=\bfcz_{n_1-1}\upharpoonright n_1$.
Therefore, by (\ref{e913}) and the properties of $\bfcg$, we see that $\bfcg\in\mathcal{F}_{n_1-1}$. Moreover, by condition (C10),
we have $w_{n_1}\in A_{n_1}$. Thus, invoking condition (C9), we conclude
\begin{equation} \label{e943}
w_{n_1} \notin \bigcap_{\bfz\in\otimes\bfcg(1)} D(\bfz).
\end{equation}
On the other hand, however, invoking condition (C10) we get that
\begin{equation} \label{e944}
w_{n_1} \in \bigcap_{\bfz\in\Delta_{n_1}} D(\bfz) \stackrel{\mathrm{(C7)}}{\subseteq} \bigcap_{\bfz\in\Gamma_{n_1}} D(\bfz)
\subseteq \bigcap_{\bfz\in\otimes\bfcg(1)} D(\bfz).
\end{equation}
This is clearly a contradiction. The proof of Lemma \ref{l91} is thus completed.


\section{Proof of the main results}

This section is devoted to the proofs of Theorem \ref{t13} and Theorem \ref{t32}. We will give the proof simultaneously
for both results following the inductive scheme outlined in \S 4. As we have already mentioned in \S 2.7, the numbers $\udhl(b|k,\ee)$ are
defined by Lemma \ref{l25}. In fact, we have $\udhl(b|k,\ee)=O_{b,\ee}(k)$.

So let $d\in\nn$ with $d\meg 1$ and $b_1,...,b_d\in\nn$ with $b_i\meg 2$ for all $i\in\{1,...,d\}$ and assume that the numbers
$\udhl(b_1,...,b_d|\ell,\eta)$ have been defined for every integer $\ell\meg 1$ and every real $0<\eta\mik 1$. Let $b_{d+1}\in\nn$
with $b_{d+1}\meg 2$ be arbitrary. We will define, recursively, the numbers $\ls(b_1,...,b_d,b_{d+1}|k,\ee)$ for every integer
$k\meg 1$ and every $0<\ee\mik 1$.

To this end notice that $\ls(b_1,...,b_d,b_{d+1}|1,\ee)=1$. Let $k\in\nn$ with $k\meg 1$ and assume that the numbers
$\ls(b_1,...,b_d,b_{d+1}|k,\eta)$ have been defined for every $0<\eta\mik 1$. Also let $0<\ee\mik 1$ be arbitrary.
From our data $b_1,...,b_d,b_{d+1},k$ and $\ee$ and the inductive assumptions, we may define the integer $K_0$, the positive
constant $r$ and the function $g:\nn\to\nn$ as in (\ref{e91}), (\ref{e92}) and (\ref{e98}) respectively. By Lemma \ref{l91}, setting
\begin{equation} \label{e101}
K_1= K_0 \lceil 2/r^2\rceil,
\end{equation}
we see that
\begin{equation} \label{e102}
\ls(b_1,...,b_d,b_{d+1}|k+1,\ee) \mik g^{(K_1)}(k+1).
\end{equation}
Having define the numbers $\ls(b_1,...,b_d,b_{d+1}|k+1,\ee)$ for every $0<\ee\mik 1$, the corresponding numbers
$\udhl(b_1,...,b_d,b_{d+1}|k+1,\ee)$ can then be estimated easily using Fact \ref{f33}.
The proofs of Theorem \ref{t13} and Theorem \ref{t32} are completed.


\section{Appendix: proof of Theorem \ref{t23}}

The family of functions $\{\phi_k:k\meg 1\}$ will be defined recursively. Notice that the case ``$k=1$" is just the finite version of the
Halpern--L\"{a}uchli Theorem for vector homogeneous trees. The existence of the corresponding function $\phi_1:\nn^3\to\nn$ follows by
\cite[Theorem 5]{So}. An essential ingredient for obtaining this bound is the work of S. Shelah \cite{Sh} on the ``Hales--Jewett" numbers
\cite{HJ}. The relation between the Halpern--L\"{a}uchli Theorem for vector homogeneous trees and the Hales--Jewett Theorem is well understood
(see, e.g., \cite{PV,To2}) and can be traced in the works of T. J. Carlson and S. G. Simpson \cite{CS} and T. J. Carlson \cite{C}.

Now let $k\in\nn$ with $k\meg 1$ and assume that the function $\phi_k$ has been defined. To define the function $\phi_{k+1}$ we need,
first, to briefly outline the general inductive step of the proof of Milliken's Theorem emphasizing, in particular, the bounds we get
from the argument. The details are fairly standard (see, e.g., \cite{To2}) and are left to the reader.

So assume that the numbers $\mil(b'_1,...,b'_{d'}|\ell,k,r')$ have been defined for every integer $d'\meg 1$, every $b'_1,...,b'_{d'}\in\nn$
with $b'_i\meg 2$ for all $i\in\{1,...,d'\}$, every integer $\ell\meg k$ and every integer $r'\meg 1$.

We fix $d\in\nn$ with $d\meg 1$ and $b_1,...,b_d\in\nn$ with $b_i\meg 2$ for all $i\in\{1,...,d\}$. Also let $r\in\nn$ with
$r\meg 1$ be arbitrary. It is convenient to introduce some notation. Specifically, for every finite vector homogeneous tree $\bfct$
with $h(\bfct)\meg k+1$ and every integer $n\mik h(\bfct)-(k+1)$ we set
\begin{equation} \label{ea1}
\strong^n_{k+1}(\bfct)=\big\{\bfcs\in\strong_{k+1}(\bfct): \bfcs(0)\in\otimes\bfct(n)\big\}.
\end{equation}
We have the following.
\begin{lem} \label{al1}
Let $n,N\in\nn$ with $N\meg k$. Also let $\bfct$ be a finite vector homogeneous tree with $b_{\bfct}=(b_1,...,b_d)$ and such that
\begin{equation} \label{ea2}
h(\bfct)\meg (n+1) + \mil\Big(\underbrace{b_1,...,b_1}_{b_1-\mathrm{times}}, ..., \underbrace{b_d,...,b_d}_{b_d-\mathrm{times}}|N,
k,r^{(\prod_{i=1}^d b_i)^{n}}\Big).
\end{equation}
Then for every $r$-coloring of $\strong^n_{k+1}(\bfct)$ there exists a vector strong subtree $\bfcs$ of $\bfct$ with
$\bfcs\upharpoonright n=\bfct\upharpoonright n$, $h(\bfcs)=(n+1)+N$ and such that for every $\bfcr,\bfcr'\in\strong^n_{k+1}(\bfcs)$
if $\bfcr(0)=\bfcr'(0)$, then $\bfcr$ and $\bfcr'$ have the same color.
\end{lem}
Now let $m\in\nn$ with $m\meg k+1$ be arbitrary and set
\begin{equation} \label{ea3}
M_1(m)=\mil(b_1,...,b_b|m,1,r).
\end{equation}
Also let $g:\nn\to\nn$ be defined by $g(n)=0$ if $n<k$ and
\begin{equation} \label{ea4}
g(n)=\mil\Big(\underbrace{b_1,...,b_1}_{b_1-\mathrm{times}}, ..., \underbrace{b_d,...,b_d}_{b_d-\mathrm{times}}|n,
k,r^{(\prod_{i=1}^d b_i)^{M_1(m)-2}}\Big)+1
\end{equation}
for every $n\meg k$. The following result is based on Lemma \ref{al1} and completes the outline of the general inductive step
of the proof of Milliken's Theorem.
\begin{lem} \label{al2}
We have
\begin{equation} \label{ea5}
\mil(b_1,...,b_d|m,k+1,r)\mik g^{(M_1(m))}(k).
\end{equation}
\end{lem}
We are now ready to define the function $\phi_{k+1}$. First we define $\zeta:\nn^3\to \nn$ by
\begin{equation} \label{ea6}
\zeta(b,m,r)=\phi_1(b,m,r).
\end{equation}
Also let $\omega:\nn^3\to\nn$ be defined by
\begin{equation} \label{ea7}
\omega(b,m,r)=r^{b^{\zeta(b,m,r)}}.
\end{equation}
Using the fact that the function $\phi_1$ belongs to the class $\mathcal{E}^6$ and elementary properties of primitive recursive functions
(see, e.g., \cite{Rose}), it is easy to check that $\zeta$ and $\omega$ belong to the class $\mathcal{E}^6$. Next we define
$f:\nn^3\to\nn$ by
\begin{equation} \label{ea8}
f(b,m,r)=\phi_k\big(b,m,\omega(b,m,r)\big)+1.
\end{equation}
By our inductive assumptions, the function $\phi_k$ belongs to the class $\mathcal{E}^{5+k}$. Hence, so does $f$. Finally, define
$\psi:\nn^4\to\nn$ recursively by the rule
\begin{equation} \label{ea9}
\left\{ \begin{array} {l} \psi(0,b,m,r)=m, \\ \psi(i+1,b,m,r)=f\big(b,\psi(i,b,m,r),r\big). \end{array}  \right.
\end{equation}
and set
\begin{equation} \label{ea10}
\phi_{k+1}(b,m,r)=\psi\big(\zeta(b,m,r),b,k,r\big).
\end{equation}
The function $\phi_{k+1}$ is the desired one. Indeed, notice that $\phi_{k+1}$ belongs to the class $\mathcal{E}^{5+k+1}$.
Moreover, using Lemma \ref{al2}, it is easily checked that
\begin{equation} \label{ea11}
\mil(b_1,...,b_d|m,k+1,r)\mik \phi_{k+1}\Big( \prod_{i=1}^d b_i^{b_i},m,r\Big).
\end{equation}
The proof of Theorem \ref{t23} is completed.


\end{document}